\documentclass[12pt,leqno,a4paper]{article}

\usepackage[english]{babel}
\usepackage{amsmath}
\usepackage{amsfonts}
\usepackage{amsthm}
\usepackage{amssymb}

\usepackage{xspace}
\usepackage{euscript}
\usepackage{graphicx}
\usepackage{amscd}
\usepackage{tabularx}

\usepackage{enumerate}
 \usepackage{epsfig} 
 \usepackage{graphics} 
\theoremstyle{plain}
\newtheorem{theo}{Theorem}[section]
\newtheorem{lem}[theo]{Lemma}
\newtheorem{prop}[theo]{Proposition}
\newtheorem{cor}[theo]{Corollary}%

\theoremstyle{definition}
\newtheorem{definition}[theo]{Definition}
\theoremstyle{remark}
\newtheorem{rem}[theo]{Remark}
\numberwithin{equation}{section}

\newcommand{\R}{\mathbb{R}}
\newcommand{\N}{\mathbb{N}}

\title{Stability Estimates for an Inverse Hyperbolic Initial Boundary Value Problem with Unknown Boundaries}

\author
{Sergio Vessella\thanks{Universit\`a degli Studi di Firenze, Italy, E-mail:
\textsf{sergio.vessella@unifi.it}}}
\date{}

\begin{document}

\setcounter{section}{0}
\setcounter{secnumdepth}{2}

\maketitle

\begin{abstract}
In this paper we prove stability estimates of logarithmic type for an inverse problem consisting in the determination of unknown portions of the boundary of a domain in $\mathbb{R}^n$, from a knowledge, in a finite time observation, of overdetermined boundary data for initial boundary value problem for anisotropic wave equation.
\medskip


\medskip

\end{abstract}

\section{Introduction} \label{sec:intro}
The reconstruction of obstacles from scattering waves has been widely investigated \cite{Col-Kr}. This approach requires enough information on the scattered amplitude and generally infinitely many boundary measurements. In many practical situations these data are not available, for example in physical situations where only transient waves are detectable and one measurement in a finite time observation is obtainable. A typical situation is described by the wave equation in a domain $\Omega$ in $\mathbb{R}^n$ ($n\geq 2$) whose boundary, assumed sufficiently smooth, consists of two non overlapping portion $\Gamma^{(a)}$ (accessible portion) and $\Gamma^{(i)}$ (inaccessible portion) where $\Gamma^{(i)}$ is a not known obstacle.

In the case in which $\Gamma^{(i)}$ is a soft obstacle the mathematical problem is represented by the following initial boundary value problem (the direct problem). Given a nontrivial function $\psi$ on $\partial\Omega\times[0,T]$, $0<T<+\infty$, such that
$$\psi=0 \mbox{ , on } \Gamma^{(i)}\times(0,T),$$
let $u$ be the (weak) solution to the following problem
\begin{equation}
\label{1-141}
\left\{\begin{array}{ll}
\partial^2_{t}u-\mbox{div}\left(A(x)\nabla_x u\right)=0, \quad \hbox{in } \Omega\times [0,T],\\[2mm]
u_{|\partial\Omega\times [0,T]}=\psi, \quad \hbox{on } \partial\Omega\times [0,T],\\[2mm]
u(\cdot,0)=\partial_tu(\cdot,0)=0, \quad \hbox{in } \Omega,
\end{array}\right.
\end{equation}
($\mbox{div}:=\sum_{j=1}^n\partial_{x_j}$) where $A(x)=\left\{a^{ij}(x)\right\}^n_{i,j=1}$ denotes a known symmetric matrix which satisfies a hypothesis of uniform ellipticity and some smoothness conditions that we shall specify in the sequel of the paper.
The inverse problem, we are interested in, is to determine $\Gamma^{(i)}$ from the knowledge of
\begin{equation}
\label{flux}
A(x)\nabla_x u(x,t)\cdot\nu\mbox{, on } \Sigma\times (0,T),
\end{equation}
where $\Sigma\subset\Gamma^{(a)}$ and $\nu$ denotes the exterior unit normal to $\Omega$.

The uniqueness for the above inverse problem has been proved in \cite{Is2}, however, in contrast to the analogues problems for elliptic equations or systems \cite{A-B-R-V}, \cite{BeVe}, \cite{CheHY}, \cite{M-R1}, \cite{M-R2}, \cite{M-R-V2} and parabolic equations \cite{CRoVe1}, \cite{CRoVe2}, \cite{DcRVe}, \cite{Ve1}, \cite{Ve2}, the stability issue in the hyperbolic context is much less studied.
In this paper we are interested in the stability issue for the inverse problem above. More precisely we are interested in the continuous dependence of $\Gamma^{(i)}$ from the Cauchy data $u$, $A\nabla_x u\cdot\nu$ on $\Sigma\times (0,T)$.
Here we prove a logarithmic stability estimate under some a priori information on the domain $\Omega$, on $\Gamma^{(i)}$, on $\psi$ and whenever $T$ is large enough, but \textit{finite and independent by the errors on the Cauchy data}. In view of John counterexample \cite{J} it is reasonable to expect that the logarithmic rate of stability is the optimal one. We are currently work on that topic

Now we describe briefly the main tools that we use to prove the stability result.

\textit{(a) Stability Estimates for Cauchy Problem and Smallness Propagation Estimates}. In order to determine the unknown portion of boundary $\Gamma^{(i)}$ it seems necessary to determine the values of $u$ from Cauchy data on $\Sigma\times (0,T)$ up to $\Gamma^{(i)}\times (0,T')$ for suitable $T'<T$. More precisely, let $\Omega_1$ and $\Omega_2$ be two domains whose boundary agree on $\Gamma^{(a)}$ and and let $u_j$ be the solutions of \eqref{1-141} for $\Omega=\Omega_j$, $j=1,2$. Denote by $G$ the connected component of $\Omega_1\cap\Omega_2$ that contains  $\Gamma^{(a)}$, we need to estimate $u_1-u_2$ in $G\times(0,T')$ in terms of the error on the Cauchy data on $\Sigma\times (0,T)$. In order to gain such estimates we use the method introduced by Robbiano in \cite{Ro_Hyp} and \cite{Ro_Cau} based on the Fourier Bros Iagolnitzer (FBI) transform defined by
\begin{equation*}
\label{defFBI}
U(x,y):=\sqrt{\frac{\mu}{2\pi}}\int^T_0e^{-\frac{\mu}{2}(iy+\tau-t)^2}\left(u_1-u_2\right)dt, \quad \hbox{for every } (x,y)\in G\times\mathbb{R},
\end{equation*}
where $\mu$ be a positive number and $\tau\in(0,T)$.

By applying such a FBI transform the wave equation is transformed in a second order elliptic equation in $G\times \mathbb{R}$ with a nonhomogeneous term $f$ depending on the final values $(u_1-u_2)(\cdot,T),\partial_t(u_1-u_2)(\cdot,T)$ and on $\mu$. Since, roughly speaking, $f$ is small whether $\mu$ and $T$ are large and $U(\cdot,0)$ is close to  $(u_1-u_2)(\cdot,\tau)$ for large $\mu$, we can apply the estimates for the Cauchy problem for elliptic equations proved in \cite{A-R-R-V} and we can obtain useful estimates of  $u_1-u_2$ in $G\times(0,T')$. We wish to stress that here we have, with respect to \cite{Ro_Hyp} and \cite{Ro_Cau}, the additional difficulty that the boundary of $G$ might be irregular.

\textit{(b) Quantitative estimates of strong unique continuation for wave equations.} For our proof it is crucial to know that the vanishing rate of $u$ near the unknown boundary $\Gamma^{(i)}$ is of polynomial type. Namely we need quantitative estimates of strong unique continuation at the interior and at the boundary. Such estimates have been proved in  \cite{Ve3} (in the present paper, Theorems \ref{5-115} and \ref{5-115Boundary}). It is exactly this property that allows us to obtain a sharp estimate of the Hausdorff distance, $d_{\mathcal{H}}\left(\overline{\Omega}_1,\overline{\Omega}_2\right)$, of the unknown domains $\Omega_1, \Omega_2$ in terms of the error on the Cauchy data (Corollary \ref{CorStimaforte}).
The use of quantitative estimate of strong unique continuation is not new in inverse problem with unknown boundaries. The first paper in which such quantitative estimates was successfully used is, in the elliptic context, \cite{A-B-R-V}. Afterwards, quantitative estimates of strong unique continuation have been used and proved also for the parabolic problems, we refer to the papers mentioned above and the review paper \cite{Ve2}. To the authors knowledge the quantitative estimate of strong unique continuation was never used before in the context of hyperbolic inverse problems.

\textit{(c) Lemma of relative graphs and sharp three sphere inequality.} At a first stage the estimate of $d_{\mathcal{H}}\left(\overline{\Omega}_1,\overline{\Omega}_2\right)$ is worse than logarithmic and, in addition, the observation time $T$ for which such estimate is available may depend on the error on the Cauchy data. In order to obtain the logarithmic stability estimate for $T$ finite and independent by the errors on the Cauchy data we combine the geometric Lemma of relative graphs (Lemma \ref{Pr4.8}) and a three sphere inequality for elliptic equations whose exponent is sharply evaluated when the radii of the three balls are close to each other (Theorem \ref{new three sphere}). This point is the most delicate part of the proof and is developed in Section \ref{step3}.

The plan of the paper is as follows.

In Section \ref{sec:notation} we will introduce the main notation and definition.

In Section \ref{maintheorem} we will state the main Theorem \ref{MainTheo}.

The Section \ref{preliminary} contains some preliminary results concerning the quantitative estimates of strong unique continuation (Subsection \ref{QEsucp}), a regularity result for hyperbolic equation (Subsection \ref{Regularity-enunciato}), some elementary estimates for the FBI transform (Subsection \ref{FBI}) and a sharp form of the three sphere inequality for elliptic equations (Subsection \ref{threeSphere}).

In Section \ref{proofmain} we prove the main Theorem \ref{MainTheo}.

In the Appendix (Section \ref{appendix}) we prove some results of Section \ref{preliminary}.

\section{Notation and Definition} \label{sec:notation}
Let $n\in\mathbb{N}$, $n\geq 2$. For any $x\in \R^n$, we will denote $x=(x',x_n)$, where $x'=(x_1,\ldots,x_{n-1})\in\R^{n-1}$, $x_n\in\R$ and $|x|=\left(\sum_{j=1}^nx_j^2\right)^{1/2}$.
Given $x\in \R^n$, $r>0$, we will use the following notation for balls and cylinders.
\begin{equation*}
   B_r(x)=\{y\in \R^n\ :\ |y-x|<r\}, \quad  B_r=B_r(0),
\end{equation*}
\begin{equation*}
   B'_r(x')=\{y'\in \R^{n-1}\ :\ |y'-x'|<r\}, \quad  B'_r=B'_r(0),
\end{equation*}
\begin{equation*}
   Q_{a,b}(x)=\{y=(y',y_n)\in \R^n\ :\ |y'-x'|<a, |y_n-x_n|<b\}, \quad Q_{a,b}=Q_{a,b}(0).
\end{equation*}
For any $x\in\mathbb{R}^n$  $x=\left(x_1\ldots,x_n\right)$ and any $r>0$ we denote by $\widetilde{x}\in\mathbb{R}^{n+1}$ the point $\widetilde{x}=\left(x_1\ldots,x_n,0\right)$, or shortly $\widetilde{x}=\left(x,0\right)$ and by $\widetilde{B}_r(\widetilde{x})$ the ball of $\mathbb{R}^{n+1}$ of radius $r$ centered at $\widetilde{x}$. For any open set $\Omega\subset\mathbb{R}^n$ and any function (smooth enough) $u$  we denote by $\nabla_x u=(\partial_{x_1}u,\cdots, \partial_{x_n})$ the gradient of $u$. Also, for the gradient of $u$ we use the notation $D_x$. If $j=0,1,2$ we denote by $D^j_x u$ the set of the derivatives of $u$ of order $j$, so $D^0_x u=u$, $D^1_x u=\nabla_x u$ and $D^2_x$ is the hessian matrix $\{\partial_{x_ix_j}u\}_{i,j=1}^n$. Similar notation are used whenever other variables occur and $\Omega$ is an open subset of $\mathbb{R}^{n-1}$ or a subset $\mathbb{R}^{n+1}$. By $H^{\ell}(\Omega)$, $\ell=0,1,2$ we denote the usual Sobolev spaces of order $\ell$, in particular we have $H^0(\Omega)=L^2(\Omega)$.

For any interval $J\subset \mathbb{R}$ and $\Omega$ as above we denote by
 \[\mathcal{W}\left(J;\Omega\right)=\left\{u\in C^0\left(J;H^2\left(\Omega\right)\right): \partial_t^\ell u\in C^0\left(J;H^{2-\ell}\left(\Omega\right)\right), \ell=1,2\right\}.\]

\bigskip

\begin{definition}[\textbf{${C}^{k,1}$ regularity of a domain}]
  \label{def:2.1}
Let $\Omega$ be a bounded domain in ${\R}^{n}$. Given $k\in\N\cup{0}$, we say that a portion $S$ of
$\partial \Omega$ is of \textit{class ${C}^{k,1}$ with
constants $\rho_{0}$, $E>0$}, if, for any $P \in S$, there
exists a rigid transformation of coordinates under which we have
$P=0$ and
\begin{equation*}
  \Omega \cap Q_{\frac{\rho_0}{E},\rho_0}=\{x=(x',x_n) \in Q_{\frac{\rho_0}{E},\rho_0}\quad | \quad
x_{n}>\varphi(x')
  \},
\end{equation*}
where $\varphi$ is a ${C}^{k,1}$ function on
$B'_{\frac{\rho_0}{E}}$ satisfying
\begin{equation*}
\|\varphi\|_{{C}^{k,1}(B'_{\rho_0/E})} \leq E\rho_{0},
\end{equation*}

\begin{equation*}
\varphi(0)=0,
\end{equation*}
and, whenever $k\geq 1$,

\begin{equation*}
\nabla_{x'}\varphi (0)=0.
\end{equation*}

\medskip
\noindent When $\partial\Omega$ is of class ${C}^{k,1}$ with
constants $\rho_{0}$, $E>0$ we also say that $\Omega$ is of class ${C}^{k,1}$ with
constants $\rho_{0}$, $E>0$. Moreover, when $k=0$ we also say that $S$ is of
\textit{Lipschitz class with constants $\rho_{0}$, $E$}.
\end{definition}

\bigskip

\begin{rem}
We use the convention of normalizing all norms in such a way that all their terms are dimensionally homogeneous. For example:
\[
\|\varphi\|_{C^{0,1}(B_{r_0}^\prime)}=\|\varphi\|_{L^{\infty}(B_{r_0}^\prime)}+r_0\|\nabla_{x'}\varphi\|_{L^{\infty}(B_{r_0}^\prime)}.
\]
Similarly, if $u\in H^m(\Omega)$, where $\Omega$ is a domain of $\mathbb{R}^n$ of class ${C}^{k,1}$ with
constants $\rho_{0}$, $E$  denoting by $D^j u$ the vector which components are the derivatives of order $j$ of the function $u$,
\begin{equation*}
\|u\|_{H^m(\Omega)}=\rho_0^{-n/2}\left(\sum_{j=0}^m
\rho_0^{2j}\int_\Omega|D_x^j u|^2\right)^{\frac{1}{2}},
\end{equation*}

\begin{equation*}
  \|u\|_{{C}^{k}(\Omega)} =\sum_{i=0}^{k}
  {r_{0}}^{i}\|D_x^{i} u\|_{{L}^{\infty}(\Omega)}.
\end{equation*}

\end{rem}

\begin{definition}
\label{Defgrafrel}\textbf{(relative graphs)}. We shall say that two bounded
domains $\Omega_1$ and $\Omega_2$ in $\mathbb{R}^n$ of class $
C^{1,1 }$ with constants $\rho_0,E$ are \textit{relative graphs} if for
any $P\in \partial \Omega_1$ there exists a rigid transformation of
coordinates under which we have $P\equiv 0$ and there exist $\varphi_{P,1},\varphi_{P,2}\in C^{1,1}\left(B_{r_0}^{\prime }\left(0\right) \right)$, where $\frac{r_0}{\rho_0}\leq 1$ depends on $E$  only, satisfying the following conditions

\begin{subequations}
\label{relgraph}
\begin{equation}
\label{relgraph-a}
\varphi _{P,1}\left( 0\right) =\left \vert \nabla_{x'} \varphi _{P,1}\left(
0\right) \right \vert=0 \quad\mbox{, } \left \vert \varphi _{P,2}\left(
0\right) \right \vert \leq \dfrac{r_{0}}{2},
\end{equation}
\begin{equation}
\label{relgraph-b}
\left \Vert\varphi_{P,i}\right \Vert_{C^{1,1}\left(
B_{r_{0}}^{\prime }\left( 0\right) \right) }\leq E\rho_0,
\end{equation}
\begin{equation}
\label{relgraph-c}
\Omega_i\cap B_{r_{0}}\left(0\right) =\left\{x\in B_{r_{0}}\left(
0\right) :x_n>\varphi_{P,i}\left( x^{\prime }\right)\right\}\mbox{, } i=1,2.
\end{equation}
\end{subequations}

\noindent We shall denote
\begin{equation}
\label{4.304}
\gamma_0\left(\Omega_1,\Omega_2\right) =\sup_{P\in \partial
\Omega_1}\left\Vert\varphi_{P,1}-\varphi_{P,2}\right \Vert_{L^{\infty
}\left( B_{r_{0}}^{\prime }\left( 0\right) \right) }
\end{equation}
and, for any $\alpha\in(0,1]$,
\begin{equation}
\label{4.304-alpha}
\gamma_{1,\alpha}\left(\Omega_1,\Omega_2\right) =\sup_{P\in \partial
\Omega_1}\left\Vert\varphi_{P,1}-\varphi_{P,2}\right \Vert_{C^{1,\alpha
}\left( B_{r_{0}}^{\prime }\left( 0\right) \right) }.
\end{equation}

\end{definition}

\bigskip

\begin{definition}
\label{Defhaus}\textbf{(Hausdorff distance). }Let $\Omega_1$ and $\Omega_2$ be bounded domains in $\mathbb{R}^n$. We call Hausdorff distance
between $\Omega_{1}$ and $\Omega_2$ the number
\begin{equation}
\label{4.350}
d_{\mathcal{H}}\left(\overline{\Omega}_1,\overline{\Omega}_2\right) =\max\left\{\sup\limits_{x\in
\Omega_1}\mbox{dist}\left( x,\overline{\Omega}_2\right)
,\sup\limits_{x\in\Omega_2}\mbox{dist}\left(x,\overline{\Omega}_1
\right)\right\}.
\end{equation}
\end{definition}

\bigskip

\begin{definition}
\label{Def4.7}\textbf{(modified distance). }Let $\Omega_1$ and $\Omega_2$ be bounded domains in $\mathbb{R}^n$. We call modified distance
between $\Omega_{1}$ and $\Omega_2$ the number
\begin{equation}
\label{4.350}
d_m\left(\overline{\Omega}_1,\overline{\Omega}_2\right) =\max\left\{\sup\limits_{x\in
\partial\Omega_1}\mbox{dist}\left( x,\overline{\Omega}_2\right)
,\sup\limits_{x\in\partial\Omega_2}\mbox{dist}\left(x,\overline{\Omega}_1
\right)\right\}.
\end{equation}
\end{definition}

For any open set $\Omega\subset\mathbb{R}^n$ and $r>0$, we shall denote
\begin{equation*}
\Omega_r=\left\{x\in\Omega: \hbox{dist}(x,\partial\Omega)> r\right\}.
\end{equation*}

We shall use the the letters $C$ to denote constants larger or equal than $1$. Sometime, for special constants or to emphasize the role that it have in the proof we will use the notation $C_0, C_1, \ldots$. The value of the constants may change from line to line, but we shall specified their dependence everywhere they appear.

\section{The Inverse Problem: The Main Theorem}\label{maintheorem}

\noindent
{\it i) A priori information on the domain.}

Given $\rho_0$, $M>0$, $E\geq 1$ we assume
\begin{subequations}
\label{1-138}
\begin{equation}
\label{1-138a}
|\Omega|\leq M\rho_0^n,
\end{equation}
\begin{equation}
\label{1-138b}
\partial\Omega \quad \hbox{ of class } C^{1,1} \hbox{ with constants } \rho_0 \hbox{ and } E,
\end{equation}
\end{subequations}
here, and in the sequel, $|\Omega|$ denotes the Lebesgue
measure of $\Omega$.

Let $\Gamma^{(a)}$ be a nonempty closed proper subset of
$\partial\Omega$ and assume that the closure of the interior part of $\Gamma^{(a)}$ in the relative
topology in $\partial\Omega$ is equal to $\Gamma^{(a)}$. In addition we assume that
\begin{equation}\label{Gammaconnected}
\mbox{Int}_{\partial\Omega} \left(\Gamma^{(a)}\right) \quad \hbox{ is connected, }
\end{equation}
and we set
\begin{equation}\label{3-138}
\Gamma^{(i)}=\partial\Omega\setminus \mbox{Int}_{\partial\Omega} \left(\Gamma^{(a)}\right) ,
\end{equation}
here and in the sequel, $\mbox{Int}_{\partial\Omega} \left(\Gamma^{(a)}\right)$ is the interior part of $\Gamma^{(a)}$ in the relative topology in $\partial\Omega$.
In the sequel we will refer to $\Gamma^{(a)}$ and $\Gamma^{(i)}$ as the \textit{accessible} and \textit{inaccessible} part of $\partial \Omega$ respectively.

Moreover denoting
\begin{equation*}
\Gamma^{(a)}_{\rho}=\left\{x\in\Gamma^{(a)}: \hbox{dist}(x,\Gamma^{(i)})\geq \rho\right\},
\end{equation*}
we assume that, for any $\rho\in(0,\rho_0]$, $\Gamma^{(a)}_{\rho}$ is a nonempty and connected set and we assume that we can select a portion $\Sigma$ satisfying for some $P_0\in\Sigma$
\begin{equation}\label{4a-138}
\partial\Omega\cap
B_{\rho_0}(P_0)\subset\Sigma \subset \Gamma^{(a)}_{\rho_0}.
\end{equation}
\noindent
\begin{rem} Observe that \eqref{1-138b} automatically
implies a lower bound on the diameter of every connected
component of $\partial\Omega$. Moreover, by combining
\eqref{1-138a} with \eqref{1-138b}, an upper bound on the diameter of
$\Omega$ can also be obtained. Note also that
\eqref{1-138a}, \eqref{1-138a} implicitly comprise an a priori upper bound
on the number of connected components of $\partial\Omega$. Finally observe that the hypotheses \eqref{1-138}-\eqref{4a-138} are satisfied in the case $\Omega=\widehat{\Omega}\setminus\overline{D}$, where $\widehat{\Omega}$ and $D$ are two open domains in $\mathbb{R}^n$ whose boundaries, $\partial\widehat{\Omega}$ and $\partial D$, are connected, $D\subset \widehat{\Omega}$, $\mbox{dist}(D,\partial \Omega)\geq 2\rho_0$ and $\widehat{\Omega}, D$ satisfy condition $\eqref{1-138}$. In addition $\Gamma^{a}=\partial\widehat{\Omega}$, $\Gamma^{i}=\partial D$ and $\Sigma$ is a portion of $\partial\widehat{\Omega}$ satisfying, for some  $P_0\in\Sigma$, the condition $\partial\widehat{\Omega}\cap
B_{\rho_0}(P_0)\subset\Sigma$.
\end{rem}

\medskip
\noindent
{\it ii) Assumptions about the boundary data.}

Let $m:=\left[\frac{n+2}{4}\right]$. Assume that $\psi$ is a function on $\partial\Omega\times [0,+\infty)$ which satisfies the following conditions

\begin{subequations}
\label{psi}
\begin{equation}
\partial_t^j\psi(\cdot,t)\in C^{1,1}(\partial\Omega)\quad \hbox{, for } j\in\{0,\cdots, 2m+4\} \hbox{, and } t\in [0,+\infty),
\end{equation}
\begin{equation}
\label{2-s13b}
 \partial_t^j\psi(\cdot,0)=0 \quad \hbox{, for } j\in\{0,\cdots, 2m+4\}\hbox{, and } t\in [0,+\infty).
\end{equation}
\end{subequations}

Denote, for $t\in[0,+\infty)$
\begin{equation}
\label{H-t}
H(t)=\sum^{2m+4}_{j=0}\rho^{j}_0\sup_{\xi\in[0,t]}\left\Vert \partial^j_\xi \psi(\cdot,\xi)\right\Vert_{C^{1,1}(\partial\Omega)}.
\end{equation}

Let $t_1\geq\rho_0$ and assume
\begin{equation}
\label{EFFEapb}
\frac{H(t_1)}
{\left\Vert\psi\right\Vert_{L^{\infty}\left(\Gamma^{(a)}\times[0,t_1]\right)}}\leq F.
\end{equation}

\medskip
\noindent
{\it iii) Assumptions about the matrix $A$.}

$A(x)=\left\{a^{ij}(x)\right\}^n_{i,j=1}$ is assumed to be a real-valued symmetric $n\times n$ matrix whose the entries are measurable function and satisfying the following conditions for given constants $\lambda\in(0,1]$, $\Lambda>0$,
\begin{subequations}
\label{1-65}
\begin{equation}
\label{1-65a}
\lambda\left\vert\xi\right\vert^2\leq A(x)\xi\cdot\xi\leq\lambda^{-1}\left\vert\xi\right\vert^2, \quad \hbox{for every } x, \xi\in\mathbb{R}^n,
\end{equation}
\begin{equation}
\label{2-65}
\left\vert A(x)-A(y)\right\vert\leq\frac{\Lambda}{\rho_0} \left\vert x-y \right\vert, \quad \hbox{for every } x, y\in\mathbb{R}^n.
\end{equation}
\end{subequations}


\bigskip

\begin{theo}\label{MainTheo}
Let $\Omega_1$, $\Omega_2$ be two
domains satisfying \eqref{1-138}. Let
$\Gamma^{(a)}_j$, $\Gamma^{(i)}_j=\partial\Omega_j\setminus \mbox{Int}_{\partial\Omega_j} (\Gamma^{(a)}_j)$, $j=1,2$, be the corresponding accessible
and inaccessible parts of their boundaries. Let us assume
$\Gamma^{(a)}_1=\Gamma^{(a)}_2=\Gamma^{(a)}$, $\Omega_1$, $\Omega_2$ lie on the
same side of $\Gamma^{(a)}$ and that \eqref{Gammaconnected}, \eqref{3-138} and \eqref{4a-138}  are satisfied.

Then there exists a constant $C$ depending on $\lambda, \Lambda, E, M$ and $F$ only such that if $T=\max\{C\rho_0,2t_1\}$ then the following holds true.

Let $u_j\in\mathcal{W}\left([0,T];\Omega\right)$ be
the solution to \eqref{1-141} when $\Omega=\Omega_j$, $j=1,2$,
and if, for a given $\varepsilon\in(0,e^{-1})$, we have
\begin{equation}
\label{3-141}
\int^{T}_{0}\int_{\Sigma} \left\vert A(x)\nabla u_1\cdot\nu-A(x)\nabla u_2\cdot\nu\right\vert^2dSdt\leq T\rho_0^{n-3}\varepsilon^2,
\end{equation}
where $dS$ is the surface element in dimension $n-1$, then we have

\begin{equation}
\label{estimate}
d_{\mathcal{H} }({\overline\Omega_1},{\overline\Omega_2})
\leq C_{\star}\rho_0|\log \varepsilon|^{-1/C_{\star}},
\end{equation}
where $C_{\star}$ depends on  $\lambda, \Lambda, E, M, F$ and the ratio $\frac{H(T)}{H(t_1)}$.
\end{theo}

We prove this Theorem in Section \ref{proofmain}.

\section{Preliminary results} \label{preliminary}

\subsection{Quantitative estimates of strong unique continuation.}\label{QEsucp}
Theorems presented in this subsection are crucial to prove Theorem \ref{MainTheo}. They are analogs of the quantitative estimates of strong unique continuation (doubling inequalities, three sphere inequality, three cylinders inequality, two-sphere one cylinder inequality at the interior and at the boundary) which are well known in the elliptic \cite{GaLi}, \cite{La}, \cite{AE} and in the parabolic context \cite{EsFeVe}, \cite{EsVe}. Theorem \ref{5-115} is basically the quantitative version of the strong unique continuation property for the self-adjoint hyperbolic equation proved by Lebeau in \cite{Le}. Theorems \ref{5-115} and \ref{5-115Boundary} have been proved in \cite{Ve3}.

Let $u\in\mathcal{W}\left([-\lambda\rho_0,\lambda\rho_0];B_{\rho_0}\right)$ be a weak solution to

\begin{equation} \label{4i-65}
\partial^2_{t}u-\mbox{div}\left(A(x)\nabla_x u\right)=0, \quad \hbox{in } B_{\rho_0}\times(-\lambda\rho_0,\lambda\rho_0).
\end{equation}
Let $r_0\in (0,\rho_0]$ and denote by

\begin{equation} \label{4ii-65}
\varepsilon_0:=\sup_{t\in (-\lambda\rho_0,\lambda\rho_0)}\left(\rho_0^{-n}\int_{B_{r_0}}u^2(x,t)dx\right)^{1/2}
\end{equation}
and

\begin{equation} \label{4iii-65}
H_0:=\left(\sum_{j=0}^2\rho_0^{j-n}\int_{B_{\rho_0}}\left\vert D_x^ju(x,0)\right\vert^2 dx\right)^{1/2}.
\end{equation}

\begin{theo}\label{5-115}
Let $A(x)$ be a real-valued symmetric $n\times n$ matrix satisfying \eqref{1-65} and let $u\in\mathcal{W}\left([-\lambda\rho_0,\lambda\rho_0];B_{\rho_0}\right)$ be a weak solution to \eqref{4i-65}. Then there exist constants $s_0\in (0,1)$ and $C\geq 1$ depending on $\lambda$ and $\Lambda$ only such that for every $r_0$ and $\rho$ satisfying $0<r_0\leq \rho\leq s_0 \rho_0$ the following inequality holds true

\begin{gather}
\label{SUCP}
\left\Vert u(\cdot,0) \right\Vert_{L^2\left(B_{\rho}\right)} \leq \frac{C\left(\rho_0\rho^{-1}\right)^{C}(H_0+e\varepsilon_0)}{\left(\theta\log \left( \frac{H_0+e\varepsilon_0}{\varepsilon_0}\right) \right)^{1/6}},
\end{gather}
where
\begin{equation}
\label{theta}
\theta=\frac{\log (\rho_0/C\rho)}{\log (\rho_0/r_0)}.
\end{equation}

\end{theo}

\bigskip

In order to state Theorem \ref{5-115Boundary} below let us introduce some notation.
Let $\varphi$ be a function belonging to $C^{1,1}\left(B^{\prime}_{\rho_0}\right)$ that satisfies

\begin{equation}
\label{phi_0}
\varphi(0)=\left\vert\nabla_{x'}\varphi(0)\right\vert=0
\end{equation}
and
\begin{equation}
\label{phi_M0}
\left\Vert\varphi\right\Vert_{C^{1,1}\left(B^{\prime}_{\rho_0}\right)}\leq E\rho_0,
\end{equation}
where

\begin{equation*}
\left\Vert\varphi\right\Vert_{C^{1,1}\left(B^{\prime}_{\rho_0}\right)}=\left\Vert\varphi\right\Vert_{L^{\infty}\left(B^{\prime}_{\rho_0}\right)}
+\rho_0\left\Vert\nabla_{x'}\varphi\right\Vert_{L^{\infty}\left(B^{\prime}_{\rho_0}\right)}+
\rho_0^2\left\Vert D_{x'}^2\varphi\right\Vert_{L^{\infty}\left(B^{\prime}_{\rho_0}\right)}.
\end{equation*}
For any $r\in (0,\rho_0]$ denote by
\[
K_{r}:=\{(x',x_n)\in B_{r}: x_n>\varphi(x')\}
\]
and
\[S_{\rho_0}:=\{(x',\varphi(x')): x' \in B^{\prime}_{\rho_0}\}.\]

Let $u\in \mathcal{W}\left([-\lambda\rho_0,\lambda\rho_0];K_{\rho_0}\right)$ be a solution to

\begin{equation} \label{4i-65Boundary}
\partial^2_{t}u-\mbox{div}\left(A(x)\nabla_x u\right)=0, \quad \hbox{in } K_{\rho_0}\times(-\lambda\rho_0,\lambda\rho_0),
\end{equation}
satisfying one of the following conditions
\begin{equation} \label{DirichBoundary}
u=0, \quad \hbox{on } S_{\rho_0}\times(-\lambda\rho_0,\lambda\rho_0),
\end{equation}

\begin{equation} \label{NeumBoundary}
A\nabla_x u\cdot \nu=0, \quad \hbox{on } S_{\rho_0}\times(-\lambda\rho_0,\lambda\rho_0),
\end{equation}
where $\nu$ denotes the outer unit normal to $S_{\rho_0}$.

Let $r_0\in (0,\rho_0]$ and denote by

\begin{equation} \label{4ii-65Boundary}
\varepsilon_0:=\sup_{t\in (-\lambda\rho_0,\lambda\rho_0)}\left(\rho_0^{-n}\int_{K_{r_0}}u^2(x,t)dx\right)^{1/2}
\end{equation}
and

\begin{equation} \label{4iii-65Boundary}
H_0:=\left(\sum_{j=0}^2\rho_0^{j-n}\int_{K_{\rho_0}}\left\vert D_x^ju(x,0)\right\vert^2 dx\right)^{1/2}.
\end{equation}

\begin{theo}\label{5-115Boundary}
Let \eqref{1-65} be satisfied. Let $u\in\mathcal{W}\left([-\lambda\rho_0,\lambda\rho_0];K_{\rho_0}\right)$ be a solution to \eqref{4i-65Boundary} satisfying \eqref{4ii-65Boundary} and \eqref{4iii-65Boundary}. Assume that $u$ satisfies either \eqref{DirichBoundary} or \eqref{NeumBoundary}. There exist constants $\overline{s}_0\in (0,1)$ and $C\geq 1$ depending on $\lambda$, $\Lambda$ and $E$ only such that for every $r_0$ and $\rho$ satisfying $0<r_0\leq \rho\leq \overline{s}_0 \rho_0$ the following inequality holds true

\begin{gather}
\label{SUCPBoundary}
\left\Vert u(\cdot,0) \right\Vert_{L^2\left(K_{\rho}\right)}\leq \frac{C\left(\rho_0\rho^{-1}\right)^{C}(H_0+e\varepsilon_0)}{\left(\widetilde{\theta}\log \left( \frac{H_0+e\varepsilon_0}{\varepsilon_0}\right)\right)^{1/6}} ,
\end{gather}
where
\begin{equation}
\label{theta}
\widetilde{\theta}=\frac{\log (\rho_0/C\rho)}{\log (\rho_0/r_0)}.
\end{equation}

\end{theo}

\subsection{A regularity result for hyperbolic equation}\label{Regularity-enunciato}
The next Theorem is a mere simplified version of a regularity result proved in \cite{Co}. For the reader convenience we give a sketch of the proof of such a result in the Appendix, Subsection \ref{Regularity-proof}.

\begin{theo}\label{boundary-Colombini}
Let $\Omega$ be a bounded domain of $\mathbb{R}^n$ that satisfies \eqref{1-138}. Let $A(x)$ be a real-valued symmetric $n\times n$ matrix satisfying \eqref{1-65}. Let $m:=\left[\frac{n+2}{4}\right]$. Assume that $\psi$ is a function on $\partial\Omega\times [0,T]$ which satisfies the condition \eqref{psi}.

Let $u\in\mathcal{W}\left([0,T];\Omega\right)$ be the solution to the problem
\begin{equation}
\label{0-s7}
\left\{\begin{array}{ll}
\partial^2_{t}u-\mbox{div}\left(A(x)\nabla_x u\right)=0, \quad \hbox{in } \Omega\times [0,T],\\[2mm]
u=\psi \quad \hbox{, on } \partial\Omega\times [0,T],\\[2mm]
u(\cdot,0)=\partial_tu(\cdot,0)=0 \quad \hbox{, in } \Omega.
\end{array}\right.
\end{equation}
Then for every $\alpha\in(0,1)$ and $t\in[0,T]$ we have $\partial^2_tu(\cdot,t)\in L^{\infty}(\Omega)$, $u(\cdot,t)\in C^{1,\alpha}(\Omega)$ and the following inequalities hold true

\begin{subequations}
\label{1-s15}
\begin{equation}
\label{1-s15a}
\sup_{t\in[0,T]}\left\Vert \partial^2_tu(\cdot,t)\right\Vert_{L^{\infty}(\Omega)}\leq C\rho_0^{-2}(T\rho_0^{-1}+1)H(T),
\end{equation}
\begin{equation}
\label{1-s15a'}
\sup_{t\in[0,T]}\left\Vert u(\cdot,t)\right\Vert_{H^{2}(\Omega)}\leq C(T\rho_0^{-1}+1)H(T),
\end{equation}
\begin{equation}
\label{1-s15b}
\sup_{t\in[0,T]}\left\Vert u(\cdot,t)\right\Vert_{C^{1,\alpha}(\Omega)}\leq C(T\rho_0^{-1}+1)H(T),
 \end{equation}
\end{subequations}
where $H(T)$ is defined by \eqref{H-t} and $C$ depends on $\alpha, n, E, M,\lambda$ and $\Lambda$ only.
\end{theo}

\subsection{Elementary estimates for the FBI transform} \label{FBI}
For the convenience of the the reader, we collect in this section some well known elementary properties of the FBI transform see also \cite{CheDY}, \cite{ChePY}, \cite{Ro_Hyp}, \cite{Ro_Cau}, \cite{RoZu}.
Let $\Omega$ be a domain of $\mathbb{R}^n$ and $T$ a positive number. Let $u\in\mathcal{W}\left([0,T];\Omega\right)$ satisfy
\begin{equation}
\label{1-56}
\left\{\begin{array}{ll}
\partial^2_{t}u-\mbox{div}\left(A(x)\nabla_x u\right)=0, \quad \hbox{in } \Omega\times [0,T],\\[2mm]
u(\cdot,0)=0, \quad \hbox{in } \Omega,\\[2mm]
\partial_tu(\cdot,0)=0, \quad \hbox{in } \Omega.
\end{array}\right.
\end{equation}
Let $\mu$ be a positive number. For a fixed $\tau\in(0,T/2]$ we denote by $U^{(\tau)}_{\mu}$ the FBI transform of $u$ defined by

\begin{equation}
\label{defFBI}
U^{(\tau)}_{\mu}(x,y):=\sqrt{\frac{\mu}{2\pi}}\int^T_0e^{-\frac{\mu}{2}(iy+\tau-t)^2}u(x,t)dt, \quad \hbox{for every } (x,y)\in\Omega\times\mathbb{R}.
\end{equation}
Observe that $U^{(\tau)}_{\mu}$ as a function of $y$ is a $C^{\infty}(\mathbb{R})$ with values in $H^2(\Omega)$.

\bigskip

The following propositions holds true.

\bigskip

\begin{prop}\label{pag56-62}
We have
\begin{gather} \label{1-60}
\left\vert D^j_{x}U^{(\tau)}_{\mu}(x,y)\right\vert\leq\\ \nonumber \leq c\mu^{1/4}e^{\frac{\mu}{2}y^2}\left(\int^T_0\left\vert D_x^j u(x,t)\right\vert^2 dt\right)^{1/2}, \quad \hbox{for a.e. } x\in\Omega, \quad \hbox{and } 0\leq j\leq2,
\end{gather}
and
\begin{gather} \label{2-59}
\left\vert U^{(\tau)}_{\mu}(x,0)-u(x,\tau)\right\vert\leq c\mu^{-1/2}\left\Vert\partial_t u(x,\cdot)\right\Vert_{L^{\infty}[0,T]}
\end{gather}
where $c$ is an absolute constant.
\end{prop}
\begin{proof}
See Subsection \ref{proofFBI}.
\end{proof}

\bigskip

\begin{prop}\label{iperb-elliptic}
Let $u\in\mathcal{W}\left([0,T];\Omega\right)$ satisfy \eqref{1-56} and let $U^{(\tau)}_{\mu}$ be defined by \eqref{defFBI}. Then $U_{\mu}$ satisfies the equation
\begin{equation} \label{ellip-63}
\partial^2_{y}U^{(\tau)}_{\mu}+\mbox{div}\left(A(x)\nabla_x U^{(\tau)}_{\mu}\right)=f^{(\tau)}_{\mu}(x,y), \quad \hbox{in } \Omega\times\mathbb{R},
\end{equation}
where
\begin{equation} \label{1-63}
f_{\mu}(x,y)=\sqrt{\frac{\mu}{2\pi}}e^{-\frac{\mu}{2}(iy+\tau-T)^2}\left(\partial_tu(x,T)-\mu(iy+\tau-T)u(x,T)\right).
\end{equation}
\end{prop}
\begin{proof}
See Subsection \ref{proofFBI}.
\end{proof}

\bigskip

\subsection{A sharp three sphere inequality for elliptic equations} \label{threeSphere}

In the following theorem we give a three sphere inequality for elliptic equations in which we take care to evaluate the exponent of the inequality when the radii of the three balls are close to each other. Except for this feature the following Theorem is quite standard and, for the convenience of the reader, we will prove it in the Appendix (Subsection \ref{App-3sphere}).

Let $\widetilde{A}(X)=\{{ \widetilde{a}^{ij}(x)}\} _{i,j=1}^{N}$, $N\geq 2$ be a real-valued symmetric $N\times N$ matrix. Assume that the entries of matrix $\widetilde{A}$ are measurable function and it satisfies

\begin{equation}
\label{0-49r-a}
\lambda_0\left\vert\xi\right\vert^2\leq \widetilde{A}(X)\xi\cdot\xi\leq\lambda_0^{-1}\left\vert\xi\right\vert^2, \quad \hbox{for every } X, \xi\in\mathbb{R}^N,
\end{equation}
where $\lambda_0\in(0,1]$.

\begin{theo}[\textbf{Three sphere inequality}] \label{new three sphere}
Let $\widetilde{r}_3$ and $\Lambda_0$ be positive numbers. Assume that $\widetilde{A}$ satisfies \eqref{0-49r-a} and
\begin{equation}
\label{0-49r-b}
\left\vert \widetilde{A}(X)-\widetilde{A}(Y)\right\vert\leq\frac{\Lambda_0}{\widetilde{r}_3} \left\vert X-Y \right\vert, \quad \hbox{for every } X, Y\in B_{\widetilde{r}_3}.
\end{equation}
Let $\widetilde{f}\in L^2(B_{\widetilde{r}_3})$ and let $u\in H^1(B_{\widetilde{r}_3})$ be a solution to
\begin{equation}
\label{equaz-49r}
Pu:=\mbox{div}(\widetilde{A}\nabla u)=\widetilde{f}\mbox{, }\quad\mbox{in } B_{\widetilde{r}_3}.
\end{equation}
Let $\widetilde{r}_1,\widetilde{r}_2,\widetilde{r}_3$ be such that $0<\widetilde{r}_1\leq \widetilde{r}_2<\widetilde{r}_3$. Let $\delta$ be such that
\begin{equation}
\label{delta-58r}
0<\delta\leq\frac{\widetilde{r}_3-\widetilde{r}_2}{2\widetilde{r}_3}\quad{ .}
\end{equation}
Denote by
\begin{equation}
\label{vartheta-58r}
\vartheta_0=\frac{\widetilde{r}_2^{-\beta}-\left[(1-\delta)\widetilde{r}_3\right]^{-\beta}}{\left[(1-2\delta)\widetilde{r}_1\right]
^{-\beta}-\left[(1-\delta)\widetilde{r}_3\right]^{-\beta}}\quad{ .}
\end{equation}
and
\begin{equation}
\label{K-58r}
C_{0}=\frac{e^{C\left[(\widetilde{r}_2\widetilde{r}_3^{-1})^{-\beta}-(1-\delta)^{-\beta}\right]}}{\delta^{4}}\quad{ ,}
\end{equation}
where $C$ depends on $\lambda_0,\Lambda_0$.

There exists $\beta_1\geq 1$ depending on $\lambda_0,\Lambda_0$ only such that if $\beta\geq \beta_1$ then the following inequality holds true
\begin{gather}
\label{1-58r}
\mathbb{}\int_{B_{\widetilde{r}_2}}\left\vert u\right\vert^2 \leq \\ \nonumber
\leq C_{0}\left(\int_{B_{\widetilde{r}_1}}\left\vert u\right\vert^2
+\widetilde{r}_3^2\int_{B_{\widetilde{r}_3}}
\left\vert\widetilde{f}\right\vert^2\right)
^{\vartheta_0}
\left(\int_{B_{\widetilde{r}_3}}\left\vert u\right\vert^2+\widetilde{r}_3^2\int_{B_{\widetilde{r}_3}}
\left\vert\widetilde{f}\right\vert^2\right)^{1-\vartheta_0}.
\end{gather}
\end{theo}

\bigskip

\section{Proof of the Main Theorem} \label{proofmain}
In order to prove Theorem \ref{MainTheo} we proceed in the following way.

Set
\[G \hbox{ the connected component of } \Omega_1\cap\Omega_2 \hbox{ whose closure contains } \Gamma^{(a)}.\]

\bigskip

\begin{description}
  \item[First step] In Proposition \ref{prop-20r} we prove that for a given $t_0>0$ there exists $T(\varepsilon)>2t_0$ such that if \eqref{3-141} is satisfied for $T=T(\varepsilon)$ and $u_j\in\mathcal{W}\left([0,T(\varepsilon)];\Omega\right)$ are the solutions to \eqref{1-141} when $\Omega=\Omega_j$, $j=1,2$ then

      \begin{equation*}
\label{1-enunciato-20r}
\sup_{t\in[0,t_0]}\left(\rho_0^{-n}\int_{\Omega_j\setminus G}u^2_j(x,t)dx\right)\leq C\omega(\varepsilon,t_0)\quad \hbox{, for } j=1,2,
\end{equation*}
where
$$\lim_{\varepsilon\rightarrow 0}\omega(\varepsilon,t_0)=0 \quad\hbox{ and }\quad \lim_{\varepsilon\rightarrow 0}T(\varepsilon)=+\infty.$$

  \item[Second step] First we prove (Proposition \ref{stima-dal-basso}) an estimate from below, in terms of the a priori information and boundary data, of the quantity $\sup\left\Vert u(\cdot,t)\right\Vert_{L^2(B_{\overline{\varrho}}(y_0))}$ where the sup is taken for $t\in [0,\overline{t}]$, $\overline{t}$ is large enough, $B_{\overline{\varrho}}(y_0)\subset\Omega$ and $\overline{\varrho}\in(0,\rho_0/2E]$. Afterwards (Proposition \ref{stimaforte}) we prove that if $t_0$ is large enough and
      \begin{equation*}
\label{3-34r}
\sup_{t\in[0,t_0]}\left(\rho_0^{-n}\int_{\Omega_j\setminus G}u^2_j(x,t)dx\right)\leq \eta^2
\end{equation*}
then
\begin{equation*}
\label{2-35r}
d_{\mathcal{H}}\left(\overline{\Omega}_1,\overline{\Omega}_2\right)\leq C\rho_0 \eta^{\alpha},
\end{equation*}
for suitable constant $C\geq 1$ and $\alpha\in(0,1)$.

\item[Third step] We conclude the proof of Theorem \ref{MainTheo}.

\end{description}

\subsection{Step 1} \label{step1}

\begin{prop}\label{prop-20r}
There exist $C\geq 1$ and $\overline{\varepsilon}, \overline{\sigma}, \vartheta_2\in (0,1]$ depending on $E, M, \lambda$ and $\Lambda$ only such that the following holds true.

Denoting
\begin{equation}
\label{3-20r}
T_{\sigma}:=\max\left\{2t_0,\sqrt{10}\rho_0\vartheta_2^{-\frac{1}{2}\sigma^{-(n+1)}}\right\},
\end{equation}
\begin{equation}
\label{Phi-20r}
\Phi(\sigma)=\sigma^{-\left(\frac{n+1}{4}\right)}(T_{\sigma}\rho_0^{-1})^{11/2}(H(T_{\sigma})+1)^2.
\end{equation}
Let us define for any $\varepsilon\in (0,\overline{\varepsilon}]$
 \begin{equation}
\label{3-epsilon20r}
T(\varepsilon):=T_{\sigma(\varepsilon)},
\end{equation}
where
\begin{equation}
\label{sigmaepsilon}
\sigma(\varepsilon):=\inf\{\sigma\in (0,\overline{\sigma}]: \quad \Phi(\sigma)\leq |\log \varepsilon|^{\frac{1}{8}}\},
\end{equation}
Let $u_j\in\mathcal{W}\left([0,T(\varepsilon)];\Omega\right)$ be the solution to \eqref{1-141} (when $T=T(\varepsilon)$) and $\Omega=\Omega_j$, $j=1,2$.

If, for a given $\varepsilon\in (0,\overline{\varepsilon}]$, we have
\begin{equation}
\label{2-20r}
\frac{1}{T(\varepsilon)\rho_0^{n-3}}\int^{T(\varepsilon)}_{0}\int_{\Sigma} \left\vert A(x)\nabla u_1\cdot\nu-A(x)\nabla u_2\cdot\nu\right\vert^2dSdt\leq \varepsilon^2
\end{equation}
then for every $t_0\in(0,T(\varepsilon)/2]$ we have
\begin{gather}
\label{1-enunciato-20r}
\sup_{t\in[0,t_0]}\left(\rho_0^{-n}\int_{\Omega_j\setminus G}u^2_j(x,t)dx\right)\leq C\omega(\varepsilon,t_0)\quad \hbox{, for } j=1,2,
\end{gather}
where
\begin{gather}
\label{omega-enunciato-20r}
\omega(\varepsilon,t_0)=(t_0\rho_0^{-1})^{6}(H(t_0))^2\left(\sigma(\varepsilon)\right)^{1/4}+\left\vert\log\varepsilon\right\vert^{-1/8}.
\end{gather}
\end{prop}

\textbf{Proof of Proposition \ref{prop-20r}.}
Let $t_0>0$. We begin by assuming only that $T\geq 2t_0$. Let $u_j\in\mathcal{W}\left([0,T];\Omega\right)$ be the solution to \eqref{1-141} when $\Omega=\Omega_j$, $j=1,2$. Let $u=u_1-u_2$, in $G\times [0,T]$ and for any positive number $\mu$ such that $\mu T^2\geq 1$ and $\tau\in(0,T/2]$ denote by $U^{(\tau)}_{\mu}$ the FBI transform of $u$ defined by

\begin{gather}
\label{defFBI-u}
U^{(\tau)}_{\mu}(x,y)=\\ \nonumber
=\sqrt{\frac{\mu}{2\pi}}\int^T_0e^{-\frac{\mu}{2}(iy+\tau-t)^2}u(x,t)dt, \quad \hbox{for every } (x,y)\in G\times\mathbb{R}.
\end{gather}
By \eqref{1-141}, \eqref{3-141} and Proposition \ref{iperb-elliptic} we have

\begin{equation}
\label{4-6r}
\left\{\begin{array}{ll}
\partial^2_{y}U^{(\tau)}_{\mu}+\mbox{div}\left(A(x)\nabla U^{(\tau)}_{\mu}\right)=f^{(\tau)}_{\mu}(x,y), \quad \hbox{in } G\times\mathbb{R},\\[2mm]
U^{(\tau)}_{\mu}(x,y)=0, \quad \hbox{for } (x,y)\in\Sigma\times\mathbb{R},\\[2mm]
\int_{\Sigma} \left\vert A(x)\nabla U^{(\tau)}_{\mu}(x,y)\cdot\nu\right\vert^2dSdt\leq C\mu^{1/2}T\rho_0^{n-3}e^{\mu y^2}\varepsilon^2,
\end{array}\right.
\end{equation}
and $C$ is an absolute constant.

By \eqref{psi},  Proposition \ref{pag56-62}, Proposition \ref{iperb-elliptic}, by Theorem \ref{boundary-Colombini} and by the elementary inequality $s^{3/2}e^{-s^2/8}\leq ce^{-s^2/10}$ we have, for every $R>0$

\begin{equation} \label{1-9r}
\left\Vert f_{\mu}\right\Vert_{L^{\infty}(G\times(-R,R))}\leq CT\rho_0^{-3}H(T)e^{\mu(R^2/2-T^2/10)},
\end{equation}
and

\begin{equation} \label{3-9r}
\left\Vert U^{(\tau)}_{\mu}\right\Vert_{L^{\infty}(G\times(-R,R))}\leq CT\rho_0^{-1}H(T)e^{\mu R^2/2},
\end{equation}
where $C$ depends on $E, M,\lambda$ and $\Lambda$ only. Here and in the sequel, we fix $\alpha=\frac{1}{2}$ in Theorem \ref{boundary-Colombini}.

Now denote by $P_1=P_0-\frac{\rho_0}{2E}\nu$, $\widetilde{P}_1=(P_1,0)$, $\rho_1=\sigma_1\rho_0$, where $\sigma_1=\frac{1}{4E\sqrt{1+E^2}}$ and denote by
\begin{equation}
\label{eta-11r}
\varepsilon_1=\frac{(\mu T^2)^{1/4}\varepsilon}{(H(T)+1)T\rho_0^{-1}}.
\end{equation}

By \eqref{4-6r}, \eqref{1-9r} and by applying \cite[Theorem 1.7]{A-R-R-V} we have
\begin{equation} \label{1-11r}
\left\Vert U^{(\tau)}_{\mu}\right\Vert_{L^{2}(\widetilde{B}_{\rho_1}(\widetilde{P}_1))}\leq CT\rho_0^{-1}H(T)e^{\mu\rho_0^2/2}\left(e^{-\mu T^2/10}+\varepsilon_1\right)^{\vartheta_1},
\end{equation}
where $\vartheta_1$, $\vartheta_1\in (0,1)$, and $C$ depend on $E, M,\lambda$ and $\Lambda$ only.

Let $\sigma\in(0,\sigma_1]$ and denote by $r=\rho_0\sigma$. Let $V_r$ be the connected component of $\Omega_{1,r}\cap\Omega_{2,r}$ whose closure contains $B_{\rho_1}(P_1))$. Moreover denote by $\omega_r=\Omega_{1,r}\setminus V_r$. We have
\begin{subequations}
\label{V-r}
\begin{equation}
\label{V-r-a}
\Omega_1\setminus G\subset\left[\left(\Omega_1\setminus\Omega_{1,r}\right)\setminus G\right]\cup \omega_r,
\end{equation}
\begin{equation}
\label{V-r-b}
\partial\omega_r=\Gamma_{1,r}\cup\Gamma_{2,r},
\end{equation}
\end{subequations}
where $$ \Gamma_{1,r}\subset\partial \Omega_{1,r}, \quad \Gamma_{2,r}\subset\partial \Omega_{2,r}\cap\partial V_r.$$

Let $z\in \Gamma_{2,r}$ be fixed. Since $V_r$ is connected, $\Gamma_{2,r}\subset\partial V_r$ and $P_1\in V_r$, there exists a continuous path
$\gamma:[0.1]\to V_r$ such that $\gamma(0)=P_1$, $\gamma(1)=z$. Let us define
$0=s_0<s_1<\ldots<s_N=1$, according to the following rule. We set $s_{k+1}=\max\{s\ |\ |\gamma(s)-x_k|=\frac{r}{2}\}$ if
$|x_k-z|>\frac{r}{2}$, otherwise we stop the process and set $N=k+1$, $s_N=1$. By \eqref{1-138a} we have $N\leq c_n M\sigma^{-n}$ where $c_n$ depends on $n$ only. Let $x_k=\gamma(s_k)$ and $\widetilde{x}_k=(x_k,0)$. The balls (of $\mathbb{R}^{n+1}$)
$\widetilde{B}_{r/4}(\widetilde{x}_k)$ are pairwise disjoint for $k=0,\ldots,N-1$ and $|\widetilde{x}_{k+1}-\widetilde{x}_k|=\frac{r}{2}$. We have that
$\widetilde{B}_{r/4}(\widetilde{x}_{k+1})\subset \widetilde{B}_{3r/4}(\widetilde{x}_k)$ and $\widetilde{B}_{r}(\widetilde{x}_{k})\subset G\times(-r,r)$ and therefore, by the three sphere inequality \eqref{1-58r} we have

\begin{gather} \label{3-12r}
\left\Vert U^{(\tau)}_{\mu}\right\Vert_{L^{2}(\widetilde{B}_{r/4}(\widetilde{x}_{k+1}))}\leq \left\Vert U^{(\tau)}_{\mu}\right\Vert_{L^{2}(\widetilde{B}_{3r/4}(\widetilde{x}_{k}))}\leq \\ \nonumber
\leq C\left(\|U^{(\tau)}_{\mu}\|_{L^{2}(\widetilde{B}_{r/4}(\widetilde{x}_{k}))}+\|f^{(\tau)}_{\mu}\|_{L^{2}(\widetilde{B}_{r}(\widetilde{x}_{k}))}\right)^{\vartheta_{\ast}}
    \left(\|U^{(\tau)}_{\mu}\|_{L^{2}(\widetilde{B}_{r}(\widetilde{x}_{k}))}+\|f^{(\tau)}_{\mu}\|_{L^{2}(\widetilde{B}_{r}(\widetilde{x}_{k}))}\right)^{1-\vartheta_{\ast}},
\end{gather}
where $C$ and $\vartheta_{\ast}$, $0<\vartheta_{\ast}<1$, depend on $E, \lambda$ and $\Lambda$ only.

Now, we denote by

\begin{equation}
\label{alpha-k}
\alpha_k=\frac{\left\Vert U^{(\tau)}_{\mu}\right\Vert_{L^{2}(\widetilde{B}_{r/4}(\widetilde{x}_k))}e^{-\mu r^2/2}}{T\rho_0^{-1}(H(T)+1)}+e^{-\mu T^2/10} \quad \hbox{for } k=0,\ldots,N,
\end{equation}
and by \eqref{3-12r}, \eqref{1-9r} and \eqref{3-9r} we have

\begin{equation}
\label{ricorsive}
\alpha_{k+1}\leq C \alpha^{\vartheta_{\ast}}_{k}\quad \hbox{for } k=0,\ldots,N-1,
\end{equation}
where $C$ and $\vartheta_{\ast}$, $0<\vartheta_{\ast}<1$, depend on $E, \lambda$ and $\Lambda$ only.
By iterating \eqref{ricorsive} we get
\begin{equation}
\label{alphaN}
\alpha_{N}\leq C^{1/1-\vartheta_{\ast}} \alpha^{\vartheta_{\ast}^{N}}_{0}.
\end{equation}
Now let us denote by $\vartheta_2=\min\left\{\vartheta_1,\vartheta_{\ast}^{c_n M}\right\}$. By \eqref{alpha-k} and \eqref{alphaN} we have
\begin{gather}
\label{1-14r}
\left\Vert U^{(\tau)}_{\mu}\right\Vert_{L^{2}(\widetilde{B}_{r/4}(\widetilde{z}))}\leq CT\rho_0^{-1}H(T)e^{\mu r^2/2}\times\\ \nonumber \times\left(\frac{\left\Vert U^{(\tau)}_{\mu}\right\Vert_{L^{2}(\widetilde{B}_{r/4}(\widetilde{P}_1))}e^{-\mu r^2/2}}{T\rho_0^{-1}(H(T)+1)}+e^{-\mu T^2/10}\right)^{\vartheta_2^{\sigma^{-n}}},
\end{gather}
where $C$ depends on $E, M, \lambda$ and $\Lambda$ only. Moreover, by applying \cite[Theorem 8.17]{GT} and by using \eqref{1-9r}, \eqref{1-11r} and \eqref{1-14r}  we have

\begin{gather}
\label{3-14r}
\left\vert U^{(\tau)}_{\mu}(z,0)\right\vert\leq CT\rho_0^{-1}(H(T)+1)e^{\mu r^2/2}\varepsilon_2,
\end{gather}
where
\begin{equation}
\label{eta-2}
\varepsilon_2=\sigma^{-\left(\frac{n+1}{2}\right)}\left(e^{-\mu T^2/10}+e^{\mu \rho_0^2/2}\left(e^{-\mu T^2/10}+\varepsilon_1\right)^{\vartheta_2}\right)^{\vartheta_2^{\sigma^{-n}}}
\end{equation}
and $C$ depends on $E, M, \lambda$ and $\Lambda$ only.

By \eqref{3-14r}, \eqref{1-s15} and \eqref{2-59}  we have

\begin{gather}
\label{15r}
\left\Vert u \right\Vert_{L^{\infty}(\Gamma_{2,r}\times[0,t_0])}\leq C (\mu T^2)^{-1/2}\left(\rho_0^{-1}T\right)^2H(T)+\\ \nonumber
+\sup_{\tau\in[0,t_0]}\left\Vert U^{(\tau)}_{\mu}(\cdot,0) \right\Vert_{L^{\infty}(\Gamma_{2,r})}\leq C(T\rho_0^{-1})^3(H(T)+1)\varepsilon_3,
\end{gather}
where
\begin{equation}
\label{eta-3}
\varepsilon_3=(\mu T^2)^{-1/2}+e^{\mu r^2/2}\varepsilon_2
\end{equation}
and $C$ depends on $E, M, \lambda$ and $\Lambda$ only.

\bigskip

By \eqref{V-r-a} and Schwarz inequality we have, for any $t\in (0,t_0]$,

\begin{gather}
\label{OmG-1}
\int_{\Omega_1\setminus G}u_1^2(x,t)dx=\\ \nonumber
=\int_{\Omega_1\setminus G}\left(\int_0^{t}\partial_\xi u_1(x,\xi) d\xi\right)^2\leq t_0\int_0^{t_0}\int_{\Omega_1\setminus G}\left\vert \partial_\xi u_1(x,\xi)\right\vert^2 dxd\xi \leq\\ \nonumber \leq t_0\int_0^{t_0}\int_{\omega_r}\left\vert \partial_\xi u_1(x,\xi)\right\vert^2 dxd\xi+t_0\int_0^{t_0}\int_{\Omega_1\setminus \Omega_{1,r}}\left\vert \partial_\xi u_1(x,\xi)\right\vert^2 dxd\xi.
\end{gather}

Now by \eqref{1-138} we have
\begin{equation}
\label{18oct}
\left\vert\Omega_1\setminus \Omega_{1,r} \right\vert\leq C \rho_0^{n}\sigma,
\end{equation}
where $C$ depends on $E$ and $M$ only.

By \eqref{1-138a} \eqref{OmG-1}, \eqref{18oct} and \eqref{1-s15a} we have, for any $t\in (0,t_0]$,
\begin{gather}
\label{OmG-2}
\rho_0^{-n}\int_{\Omega_1\setminus G}u_1^2(x,t)dx\leq\\ \nonumber
 \leq t_0\rho_0^{-n}\int_0^{t_0}\int_{\omega_r}\left\vert \partial_\xi u_1(x,\xi)\right\vert^2dxd\xi+C\left(t_0\rho_0^{-1}\right)^6H(t_0)^2\sigma,
\end{gather}
where $C$ depends on $E, M, \lambda$ and $\Lambda$ only.

Now, in order to estimate from above the integral on the right hand side of \eqref{OmG-2} we multiply both the side of the equation $\partial^2_{t}u_1-\mbox{div}\left(A(x)\nabla u_1\right)=0$ by $\partial_t u_1$ and integrate over $\omega_r$ and by integration by parts we have, for every $\xi\in[0,t_0]$,

\begin{gather}
\label{1-144}
\frac{1}{2}\int_{\omega_r}\left(\left\vert \partial_\xi u_1(x,\xi)\right\vert^2+A(x)\nabla u_1(x,\xi)\cdot \nabla u_1(x,\xi)\right)dx=\\ \nonumber
=\int_0^{\xi}\int_{\Gamma_{1,r}}\left(A(x)\nabla u_1(x,t)\right)\partial_t u_1(x,t)dSdt+\int_0^{\xi}\int_{\Gamma_{2,r}}\left(A(x)\nabla u_1(x,t)\right)\partial_t u_1(x,t)dSdt:=J_1+J_2.
\end{gather}

\bigskip

\textbf{ Estimate of $J_1$}.

By Schwarz inequality and by \eqref{1-65a} we have
\begin{gather}
\label{3/2-1r}
|J_1|\leq\lambda^{-1}\left(\int_0^{\xi}\int_{\Gamma_{1,r}}\left\vert\nabla u_1\right\vert^2dSdt\right)^{1/2}\left(\int_0^{\xi}\int_{\Gamma_{1,r}}\left\vert\partial_t u_1(x,t)\right\vert^2dSdt\right)^{1/2}.
\end{gather}

By \eqref{3/2-1r} and \eqref{1-s15b} we have, for every $\xi\in[0,t_0]$,
\begin{gather}
\label{2-1r}
|J_1|\leq C\left((t_0\rho_0^{-1}+1)t_0\rho_0^{n-3}\right)^{1/2}H(t_0)\left(\int_0^{\xi}\int_{\Gamma_{1,r}}\left\vert\partial_t u_1(x,t)\right\vert^2dSdt\right)^{1/2},
\end{gather}
where $C$ depends on $E, M, \lambda$ and $\Lambda$ only.

Now, by interpolation inequality we have

\begin{gather*}
\left\Vert\partial_tu_1\right\Vert_{L^{\infty}(\Gamma_{1,r}\times[0,t_0])}\leq C\left\Vert u_1\right\Vert^{1/2}_{L^{\infty}(\Gamma_{1,r}\times[0,t_0])}
\left\Vert\partial^2_tu_1\right\Vert^{1/2}_{L^{\infty}(\Gamma_{1,r}\times[0,t_0])},
\end{gather*}
where $C$ is an absolute constant, hence by using \eqref{2-1r}, \eqref{1-s15} and recalling that $u_1=0$ on $\Gamma_{1}\times[0,T]$ we obtain
\begin{gather}
\label{1-2r}
|J_1|\leq C(t_0\rho_0^{-1})^{5/2}\rho_0^{n-2}\left(H(t_0)\right)^2\sigma^{1/4},
\end{gather}
where $C$ depends on $E, M, \lambda$ and $\Lambda$ only.

\bigskip

\textbf{Estimate of $J_2$}.

By Schwarz inequality, \eqref{1-65a} and \eqref{1-s15b} we have, for every $\xi\in[0,t_0]$,
\begin{gather}
\label{1-3r}
|J_2|\leq C\left(t^2_0\rho_0^{n-4}\right)^{1/2}H(t_0)
\left(\int_0^{\xi}\int_{\Gamma_{2,r}}\left\vert\partial_t u_1(x,t)\right\vert^2dSdt\right)^{1/2},
\end{gather}
where $C$ depends on $ E, M, \lambda$ and $\Lambda$ only.

By the triangle inequality and taking into account that $u=u_1-u_2$ on $\Gamma_{2,r}\times [0,T]$ we have, for every $\xi\in[0,t_0]$,

\begin{gather}
\label{2-3r}
\left(\int_0^{\xi}\int_{\Gamma_{2,r}}\left\vert\partial_t u_1(x,t)\right\vert^2dSdt\right)^{1/2}\leq  \\ \nonumber
\leq C\left(t_0\rho_0^{n-1}\right)^{1/2}\left(\left\Vert\partial_tu\right\Vert_{L^{\infty}(\Gamma_{2,r}\times[0,t_0])}+\left\Vert\partial_tu_2
\right\Vert_{L^{\infty}(\Gamma_{2,r}\times[0,t_0])}\right),
\end{gather}
where $C$ depends on $E$ and $M$ only.

Arguing as in the estimate of $J_1$, by \eqref{15r}, \eqref{1-3r}, \eqref{2-3r} and \eqref{1-s15a} we have
\begin{gather}
\label{J2}
|J_2|\leq  C\rho_0^{n-2}(t_0\rho_0^{-1})^{5/2}(H(t_0))^2\sigma^{1/4}+ \\ \nonumber
+C\rho_0^{n-2}(T\rho_0^{-1})^3(H(T)+1)^2\varepsilon_3^{1/2},
\end{gather}
where $C$ depends on $ E, M, \lambda$ and $\Lambda$ only.

By \eqref{OmG-2}, \eqref{1-144}, \eqref{1-2r} and \eqref{J2} we have, for every $t\in (0,t_0]$,
\begin{gather}
\label{1-16r}
\rho_0^{-n}\int_{\Omega_1\setminus G}u_1^2(x,t)dx\leq  C(t_0\rho_0^{-1})^{6}(H(t_0))^2\sigma^{1/4}+ \\ \nonumber
+C(T\rho_0^{-1})^5(H(T)+1)^2\varepsilon_3^{1/2},
\end{gather}
where $C$ depends on $\alpha, E, M, \lambda$ and $\Lambda$ only.
In order estimate from above the right hand side of \eqref{1-16r} first we assume
\begin{equation*}
\label{varepsilon1-1}
\varepsilon\leq e^{-5}.
\end{equation*}
Let $\mu$ and $T$ be such that
\begin{equation}
\label{esse}
\mu T^2=\frac{1}{5}\left\vert\log\varepsilon\right\vert.
\end{equation}
By \eqref{eta-11r}, \eqref{esse} we have trivially

\begin{equation}
\label{harm-1}
e^{-\mu T^2/10}+\varepsilon_1\leq c\varepsilon^{1/2},
\end{equation}
where $c$ is an absolute constant. Hence, taking into account \eqref{eta-2} and \eqref{eta-3} we have
\begin{gather}
\label{harm-2}
\varepsilon_3^{1/2}\leq (\mu T^2)^{-1/4}+e^{ \mu T^2(\rho_0T^{-1})^2\sigma^2/4}\varepsilon_2^{1/2}\leq\\
\nonumber \leq (\mu T^2)^{-1/4}+C\sigma^{-\left(\frac{n+1}{4}\right)}e^{\frac{1}{4}\mu K(\sigma,T)},
\end{gather}
where
\begin{equation}
\label{Kappa}
K(\sigma, T)=2\rho_0^2\sigma^2-\frac{T^2}{5}\vartheta_2^{\sigma^{-n-1}}.
\end{equation}

Let us choose

\begin{equation}
\label{Tsigma}
T=T_{\sigma}=\max\left\{2t_0,\sqrt{10}\rho_0\vartheta_2^{-\frac{1}{2}\sigma^{-(n+1)}}\right\},
\end{equation}
and we have $K(\sigma, T_{\sigma})\leq-3\rho_0^2$. Hence by \eqref{harm-2} and \eqref{Tsigma} we have

\begin{gather}
\label{harm-3}
\varepsilon_3^{1/2}\leq C \sigma^{-\left(\frac{n+1}{4}\right)}\left(T_{\sigma}\rho_0^{-1}\right)^{1/2}\left\vert\log\varepsilon\right\vert^{-1/4},
\end{gather}
where $C$ depends on $E, M, \lambda$ and $\Lambda$ only.
By \eqref{1-16r} and \eqref{harm-3} we have, for every $t\in (0,t_0]$ and $\sigma\in (0, \sigma_1]$,
\begin{gather}
\label{1-19r}
\rho_0^{-n}\int_{\Omega_1\setminus G}u_1^2(x,t)dx\leq \\ \nonumber
\leq C\left((t_0\rho_0^{-1})^{6}(H(t_0))^2\sigma^{1/4}+\Phi(\sigma)\left\vert\log\varepsilon\right\vert^{-1/4}\right),
\end{gather}
where $C$ depends on $E, M, \lambda$ and $\Lambda$ only and $\Phi(\sigma)$ is defined by \eqref{Phi-20r}.

Let $\overline{\sigma}=\min\{\sigma_1,(2n|\log\vartheta_2|)^{1/(n+1)}\}$, by \eqref{Tsigma} we have that $\Phi$ is a decreasing function in $(0,\overline{\sigma}]$, so that $\min_{(0,\overline{\sigma}]}\Phi=\Phi(\overline{\sigma})$. Now, let us denote by $\overline{\varepsilon}=\min\{e^{-5},e^{-(\Phi(\overline{\sigma}))^8}\}$ and for any $\varepsilon\in (0,\overline{\varepsilon}]$ let us choose $\sigma=\sigma(\varepsilon)$ where $\sigma=\sigma(\varepsilon)$ is defined by \eqref{sigmaepsilon}. By \eqref{1-19r} we have
\begin{gather}
\label{1-20r}
\rho_0^{-n}\int_{\Omega_1\setminus G}u_1^2(x,t)dx\leq C\omega(\varepsilon,t_0),
\end{gather}
where $C$ depends on $E, M, \lambda$ and $\Lambda$ only and $\omega(\varepsilon,t_0)$ is defined by \eqref{omega-enunciato-20r}.$\square$

\subsection{Step 2} \label{step2}
\begin{prop}\label{stima-dal-basso}
Let $\overline{\varrho}\in(0,\rho_0/2E]$ and let $y_0\in \Omega$ be such that $B_{\overline{\varrho}}(y_0)\subset\Omega$.
Assume that $u$ is solution to \eqref{1-141}. Then there exists a constant $C_F$, $C_F\geq2$, depending on $E, M, \lambda,\Lambda,\overline{\varrho}\rho_0^{-1}$ and $F$ only such that if
$\overline{t}\geq t_{\ast}:=\max\{C_F\rho_0,2t_1\}$  then the following inequality holds true
\begin{equation}
\label{stimabasso}
\overline{t}\rho_0^{-1}H(\overline{t})e^{-\mathcal{F}(\overline{t})}
\leq
\sup_{t\in[0,\overline{t}]}\left\Vert u(\cdot,t)\right\Vert_{L^2(B_{\overline{\varrho}}(y_0))},
\end{equation}
where
\begin{equation}
\label{F-CALL}
\mathcal{F}(\overline{t})=\left(\frac{C_F(\overline{t}\rho_0^{-1})^3H(\overline{t})}
{H(t_1)}\right)^2.
\end{equation}
\end{prop}

\bigskip

\begin{proof}
For any number $\overline{t}$ such that $\overline{t}\geq 2t_1$ let us denote
\begin{equation}
\label{eta-45r}
\eta=\sup_{t\in[0,\overline{t}]}\left\Vert u(\cdot,t)\right\Vert_{L^2(B_{\overline{\varrho}}(y_0))}.
\end{equation}
Let $\left(x_0,\overline{\tau}\right)\in\Gamma^{(a)}\times[0,t_1]$ be such that
\begin{equation}
\label{maxpsi}
\left\vert\psi \left(x_0,\overline{\tau}\right)\right\vert=\left\Vert\psi\right\Vert_{L^{\infty}\left(\Gamma^{(a)}\times[0,t_1]\right)}.
\end{equation}
Let $\delta\in\left(0,\frac{1}{4}\right]$ be a number that we will choose later and let $x_{\delta}=x_0-4\delta\overline{\varrho}\nu(x_0)$. By Proposition \ref{boundary-Colombini} and by \eqref{maxpsi} we have

\begin{gather}
\label{1-59r}
\left\Vert\psi\right\Vert_{L^{\infty}\left(\Gamma^{(a)}\times[0,t_1]\right)}\leq \left\vert u \left(x_0,\overline{\tau}\right)-u\left(x_{\delta},\overline{\tau}\right)\right\vert+\left\vert u\left(x_{\delta},\overline{\tau}\right)\right\vert\leq \\ \nonumber
\leq C_0t_1\overline{\varrho}\rho_0^{-2}H(t_1)\delta+\left\vert u\left(x_{\delta},\overline{\tau}\right)\right\vert
\end{gather}
where $C_0$ depends on $E, M, \lambda$ and $\Lambda$ only.
Now let us choose
\begin{equation*}
\label{deltabar}
\overline{\delta}=\min\left\{\frac{1}{4},\frac{\rho_0}{2C_0t_1F}\right\}
\end{equation*}
and by \eqref{1-59r} we have

\begin{gather}
\label{1-59r-0}
\frac{1}{2}\left\Vert\psi\right\Vert_{L^{\infty}\left(\Gamma^{(a)}\times[0,t_1]\right)}\leq\left\vert u\left(x^{(\overline{\delta})},\overline{\tau}\right)\right\vert.
\end{gather}

Now we estimate from above in terms of $\eta$ the right hand side of \eqref{1-59r-0}. In order to get such an estimate we proceed similarly to Proposition \ref{prop-20r}. For any positive number $\mu$ such that $\mu \overline{t}^2\geq 1$ and $\tau\in(0,\overline{t}/2]$ denote by $U^{(\tau)}_{\mu}$ the FBI transform of $u$ defined by

\begin{equation}
\label{1-46r}
U^{(\tau)}_{\mu}(x,y):=\sqrt{\frac{\mu}{2\pi}}\int^{\overline{t}}_0e^{-\frac{\mu}{2}(iy+\tau-t)^2}u(x,t)dt, \quad \hbox{for every } (x,y)\in \Omega\times\mathbb{R}.
\end{equation}

Denote by $x_1=x_0-\overline{\varrho}\nu(x_0)$ where $\nu(x_0)$ is the exterior unit normal to $\partial\Omega$ in $x_0$.
Since
\begin{equation*}
\label{Umu-46r}
\left\{\begin{array}{ll}
\partial^2_{y}U^{(\tau)}_{\mu}+\mbox{div}\left(A(x)\nabla U^{(\tau)}_{\mu}\right)=f^{(\tau)}_{\mu}(x,y), \quad \hbox{in } \Omega\times\mathbb{R},\\[4mm]
\int_{B_{\overline{\varrho}}(y_0)} \left\vert U^{(\tau)}_{\mu}(x,y)\right\vert^2dx\leq c\mu^{1/2}\overline{t}\rho_0^{n}e^{\mu y^2}\eta^2,
\end{array}\right.
\end{equation*}
by arguing as in Proposition \ref{prop-20r} we get

\begin{gather}
\label{0-48r}
\left\Vert U^{(\tau)}_{\mu}\right\Vert_{L^{2}(\widetilde{B}_{\overline{\varrho}/4}(\widetilde{x}_1))}\leq C\overline{t}\rho_0^{-1}H(\overline{t})e^{\mu \overline{\varrho}^2/2}\times\\ \nonumber \times\left(\frac{\left\Vert U^{(\tau)}_{\mu}\right\Vert_{L^{2}(\widetilde{B}_{\overline{\varrho}/4}(\widetilde{y}_0))}e^{-\mu \overline{\varrho}^2/2}}{\overline{t}\rho_0^{-1}(H(\overline{t})+1)}+e^{-\mu \overline{t}^2/10}\right)^{\vartheta^{\ast}_{1}},
\end{gather}
and
\begin{gather}
\label{1-48r}
\left\Vert U^{(\tau)}_{\mu}\right\Vert_{L^{2}(\widetilde{B}_{\overline{\varrho}/4}(\widetilde{x}_1))}\leq C(\overline{t}\rho_0^{-1}H(\overline{t})+2\eta)e^{\mu \overline{\varrho}^2/2}\left((\mu \overline{t}^2)^{3/2}e^{-\mu \overline{t}^2/8}+\eta_1\right)^{\vartheta^{\ast}_{1}},
\end{gather}
where
\begin{equation}
\label{eta-3DICEM}
\eta_1=\frac{(\mu \overline{t}^2)^{1/4}\eta}{H(\overline{t})\overline{t}\rho_0^{-1}+2\eta},
\end{equation}
$\vartheta^{\ast}_{1}\in(0,1)$ depends on $E, M, \lambda, \Lambda$ and $\overline{\varrho}$ only and $C$ depends on $E, M, \lambda$ and $\Lambda$ only.

By \eqref{1-s15}, \eqref{2-59}, \eqref{1-63}, \eqref{1-59r-0} and by applying \cite[Theorem 8.17]{GT} we have
\begin{gather}
\label{1-59r-1}
\frac{1}{2}\left\Vert\psi\right\Vert_{L^{\infty}\left(\Gamma^{(a)}\times[0,t_1]\right)}\leq
\left\vert u\left(x_{\overline{\delta}},\overline{\tau}\right)-U^{(\overline{\tau})}_{\mu}\left(x_{\overline{\delta}},0\right)\right\vert+\left\vert U^{(\overline{\tau})}_{\mu}\left(x_{\overline{\delta}},0\right)\right\vert \leq
 \\ \nonumber
\leq C(\mu \overline{t}^2)^{-1/2}(\overline{t}\rho_0^{-1})^2\overline{t}H(\overline{t})+\left\Vert U^{(\overline{\tau})}_{\mu}\right\Vert_{L^{\infty}(\widetilde{B}_{\overline{\varrho}(1-3\overline{\delta})}(\widetilde{x}_1))}\leq \\ \nonumber
\leq C\left((\mu \overline{t}^2)^{-1/2}+e^{\mu(\overline{\varrho}^2/2-\overline{t}^2/10)}\right)(\overline{t}\rho_0^{-1})^3H(\overline{t})+\frac{C'}{\overline{\delta}^{(n+1)/2}}
\left\Vert U^{(\overline{\tau})}_{\mu}\right\Vert_{L^{2}(\widetilde{B}_{\overline{\varrho}(1-2\overline{\delta})}(\widetilde{x}_1))},
\end{gather}
where $C$ depends on $E, M, \lambda$ and $\Lambda$ only and where $C'$ depends on $\lambda$ only.

Now let us apply the three sphere inequality \eqref{1-58r} with $r_1=\frac{\overline{\varrho}}{4}$, $r_2=\overline{\varrho}(1-2\overline{\delta})$ and $r_3=\overline{\varrho}$. By \eqref{1-63} and \eqref{1-48r} we have
\begin{equation}
\label{3sphere-62r}
\left\Vert U^{(\overline{\tau})}_{\mu}\right\Vert_{L^{2}(\widetilde{B}_{\overline{\varrho}(1-2\overline{\delta})}(\widetilde{x}_1))}\leq C\left((\overline{t}\rho_0^{-1})^3H(\overline{t})+2\eta\right)e^{\mu \overline{\varrho}^2/2}\left(e^{-\mu \overline{t}^2/10}+\eta_1\right)^{\vartheta^{\ast}_{2}}
\end{equation}
where $\vartheta^{\ast}_{2}$, $\vartheta^{\ast}_{2}\in(0,1)$, and $C$ depend on $E, M, \lambda,\Lambda,\overline{\varrho}\rho_0^{-1}$ and $F$ only.

By \eqref{1-59r-1} and \eqref{3sphere-62r} we have
\begin{gather}
\label{62r-1}
\left\Vert\psi\right\Vert_{L^{\infty}\left(\Gamma^{(a)}\times[0,t_1]\right)}\leq
 C(\overline{t}\rho_0^{-1}H(\overline{t})+2\eta)\left((\mu \overline{t}^2)^{-1/2}+e^{\mu \overline{\varrho}^2/2}\left(e^{-\mu \overline{t}^2/10}+\eta_1\right)^{\vartheta^{\ast}_{2}}\right),
\end{gather}
where $C$ depends on $E, M, \lambda,\Lambda,\overline{\varrho}$ and $F$ only.
Now if $\overline{t}\geq\max\left\{\sqrt{10}(\vartheta^{\ast}_{2})^{-1/2}\overline{\varrho},2\rho_0, 2t_1\right\}$ then \eqref{eta-3DICEM} and \eqref{62r-1} give

\begin{gather}
\label{62r-3}
\left\Vert\psi\right\Vert_{L^{\infty}\left(\Gamma^{(a)}\times[0,t_1]\right)}
\leq C\left((\overline{t}\rho_0^{-1})^3H(\overline{t})+2\eta\right)\times \\ \nonumber
\times\left((\mu \overline{t}^2)^{-1/2}+(\mu \overline{t}^2)^{\vartheta^{\ast}_{2}/4}e^{\mu \vartheta^{\ast}_{2}\overline{t}^2/20}\left(\frac{\eta}{\overline{t}\rho_0^{-1}H(\overline{t})+2\eta}\right)^{\vartheta^{\ast}_{2}}\right),
\end{gather}
where $C$ depends on $E, M, \lambda,\Lambda,\overline{\varrho}\rho_0^{-1}$ and $F$ only.

Now let us choose
$$\mu=\frac{10}{\overline{t}^2}\left\vert\log\left(\frac{\eta}{\overline{t}\rho_0^{-1}H(\overline{t})+2\eta}\right)\right\vert$$
and by \eqref{62r-3},taking into account that $\eta\leq CH(\overline{t})$, we get
\begin{gather}
\label{62r-4}
\left\Vert\psi\right\Vert_{L^{\infty}\left(\Gamma^{(a)}\times[0,t_1]\right)}
\leq C(\overline{t}\rho_0^{-1})^2H(\overline{t})\left\vert\log\left(\frac{\eta}{CT_1\rho_0^{-1}H(\overline{t})}\right)\right\vert^{-1/2},
\end{gather}
where $C$ depends on $E, M, \lambda,\Lambda,\overline{\varrho}\rho_0^{-1}$ and $F$ only. By \eqref{62r-4} the thesis follows.
\end{proof}

Now we recall the following Lemma that was proved in \cite[Lemma 8.1]{A-B-R-V}.
\begin{lem}[\textbf{relative graphs}]
\label{Pr4.8}
Let $\Omega_1$ and $\Omega_2$ be bounded domains in $\mathbb{R}^n$ of class $C^{1,1}$ with constants $\rho_0$, $E$ and
satisfying $\left\vert\Omega_j\right\vert\leq M\rho_0^n$, $j=1,2$. There exist
numbers $d_0$, $\overline{\rho}_0\in\left(0,\rho_0\right] $ such that $\frac{
d_0}{\overline{\rho}_0}$ and $\frac{\overline{\rho}_0}{\rho_0}$ depend on $E$
only, and such that if we have
\begin{equation}
\label{4.360}
d_{\mathcal{H}}\left(\overline{\Omega}_1,\overline{\Omega}_2\right)
\leq d_{0},
\end{equation}
then the following facts hold true

\noindent i) $\Omega_1$ and $\Omega_2$ are relative graphs and

\begin{equation}
\label{4.361-a}
\gamma_0 \left(\Omega_1,\Omega_2\right) \leq Cd_{\mathcal{H}}\left(
\overline{\Omega}_1,\overline{\Omega}_2\right),
\end{equation}
where $C$ depends $E$ only,
\begin{equation}
\label{4.361-b}
\gamma_{1,\alpha}\left(\Omega_1,\Omega_2\right) \leq C\rho_0^{1+\alpha \over 2}
\left(d_{\cal H}(\overline{\Omega}_1,\overline{\Omega}_2)\right)
^{1-\alpha \over 2},\quad\hbox{for every }
\alpha\in(0,1),
\end{equation}
where $C$ depends $E$ and $\alpha$ only,

\noindent ii) there exists an absolute positive constant $c$ such that
\begin{equation}
 \label{4.365}
d_{\mathcal{H}}\left(\overline{\Omega}_1,\overline{\Omega}_2\right)
\leq cd_{m}\left(\overline{\Omega}_1,\overline{\Omega}_2\right),
\end{equation}
iii) $\Omega_1\cap\Omega_2$ is a domain of Lipschitz class with constants $\overline{\rho}_0$, $L$, where $\overline{\rho}_0$ is as
above and $L>0$ depends on $E$ only.
\end{lem}

\bigskip

\begin{prop}
\label{stimaforte}
There exist constants $C_F$ and $C$ depending on $E, M, \lambda,\Lambda$ and $F$ only and on $E, M, \lambda$ $\Lambda$ only respectively, such that if $t_0\geq t_{\ast}+\lambda\rho_0$ where $t_{\ast}$ is introduced in Proposition \ref{stima-dal-basso} and if
\begin{equation}
\label{3-34r}
\sup_{t\in[0,t_0]}\left(\rho_0^{-n}\int_{\Omega_j\setminus G}u^2_j(x,t)dx\right)\leq \eta^2
\end{equation}
then
\begin{equation}
\label{2-35r}
d_{\mathcal{H}}\left(\overline{\Omega}_1,\overline{\Omega}_2\right)\leq C\rho_0\left(\frac{\eta}{\overline{t}_0\rho_0^{-1}H(\overline{t}_0)}\right)^{1/K_0},
\end{equation}
where
\begin{equation}
\label{1-38r}
K_0=e^{C\mathcal{F}(\overline{t}_0)}.
\end{equation}
$$\overline{t}_0=t_0-\lambda\rho_0$$
and $\mathcal{F}(\overline{t}_0)$ is defined by \eqref{F-CALL}.
\end{prop}
\begin{proof}

First we prove the following inequality
\begin{equation}
\label{1-39r}
d_{m}\left(\overline{\Omega}_1,\overline{\Omega}_2\right)\leq C\rho_0\left(\frac{\eta}{\overline{t}_0\rho_0^{-1}H(\overline{t}_0)}\right)^{1/K_0},
\end{equation}
where $d_m(\overline{\Omega}_1,\overline{\Omega}_2)$ is the quantity introduced in Definition \ref{Def4.7}.

For the sake of brevity let us denote $d_m=d_{m}\left(\overline{\Omega}_1,\overline{\Omega}_2\right)$
Let us assume, with no loss of generality, that there exists
$x_0\in \Gamma^{(i)}_1\subset \partial\Omega_1$ such that
$\hbox{dist}(x_0,\Omega_2)=d_m$.

By \eqref{3-34r} we have trivially

\begin{equation}
\label{2-35r-1}
\sup_{t\in[0,t_0]}\left(\rho_0^{-n}\int_{\Omega_1\cap B_{d_m}(x_0)}u^2_1(x,t)dx\right)\leq \eta^2
\end{equation}
let us distinguish the following two cases

i) $d_m\leq \frac{1}{2}\overline{s}_0\rho_0$,

ii) $d_m> \frac{1}{2}\overline{s}_0\rho_0$,

\noindent where $\overline{s}_0$, $\overline{s}_0\in(0,1)$, is defined in Theorem \ref{5-115Boundary} and depends on $E, \lambda,$ and $\Lambda$ only.

In case i), by applying Theorem \ref{5-115Boundary} with $r_0=d_m$ and $\rho=\frac{\overline{s}_0\rho_0}{2}$ we have

\begin{gather}
\label{2-37r}
\sup_{t\in[0,t_0-\lambda\rho_0]}\left\Vert u_1(\cdot,t) \right\Vert_{L^2\left(B_{\overline{s}_0\rho_0/2}(x_0)\cap\Omega_1\right)} \leq \\ \nonumber
 \leq C\left(\rho_0^{-1}t_0H(t_0)\right)\left(\theta_1\log \left(\frac{\rho_0^{-1}t_0H(t_0)}{\eta}\right)\right)^{-1/6},
\end{gather}
where
\begin{equation}
\label{theta1}
\theta_1=\frac{1}{C\log (\rho_0/d_m)}.
\end{equation}
and $C$ depends on $E, M,\lambda$ and $\Lambda$ only.

Now let us introduce the following notation: $s^{\ast}=\min\left\{\frac{\overline{s}_0}{4},\frac{1}{2E}\right\}$ and $y_0=x_0-s^{\ast}\rho_0\nu(x_0)$, $\overline{t}_0=t_0-\lambda\rho_0$. We have $B_{s^{\ast}\rho_0/2}(y_0)\subset B_{s_0\rho_0/2}(x_0)\cap\Omega_1$. Let us assume that $t_0\geq\max\{2C_{F}\rho_0,2t_1\}$ where $C_{F}$ is defined in Proposition \ref{stima-dal-basso}. By \eqref{2-37r} and Proposition \ref{stima-dal-basso} we have

\begin{gather}
\label{2-37r-1}
\overline{t}_0\rho_0^{-1}H(\overline{t}_0)e^{-\mathcal{F}(\overline{t}_0)}\leq \\ \nonumber
\leq C\left(\rho_0^{-1}t_0H(t_0)\right)\left(\theta_1\log \left(\frac{\rho_0^{-1}t_0H(t_0)}{\eta}\right)\right)^{-1/6},
\end{gather}
where $C$ depends on $E, M,\lambda$ and $\Lambda$ only and $\mathcal{F}(\overline{t}_0)$ is defined by \eqref{F-CALL}.

By \eqref{theta1} and \eqref{2-37r-1} it is easy to get
\begin{equation}
\label{2-38r}
d_m\leq \rho_0\left(\frac{\eta}{\rho_0^{-1}t_0H(t_0)}\right)^{1/K_0},
\end{equation}
where $K_0$ is defined in \eqref{1-38r}.

Consider now case ii). Since we have $B_{s^{\ast}\rho_0/2}(y_0)\subset B_{s_0\rho_0/2}(x_0)\cap\Omega_1$ we get
\begin{gather}
\label{star-36r}
\overline{t}_0\rho_0^{-1}H(\overline{t}_0)e^{-\mathcal{F}(\overline{t}_0)}\leq \sup_{t\in[0,t_0-\lambda\rho_0]}\left\Vert u_1(\cdot,t) \right\Vert_{L^2\left(B_{\vartheta^{\ast}_{1}\rho_0/2}(y_0)\right)}\leq \\ \nonumber
\leq\sup_{t\in[0,t_0-\lambda\rho_0]}\left\Vert u_1(\cdot,t) \right\Vert_{L^2\left(B_{s_0\rho_0/2}(x_0)\cap\Omega_1\right)}\leq \eta.
\end{gather}
Hence
\begin{equation}
\label{trivially}
1\leq \frac{e^{\mathcal{F}(\overline{t}_0)}\eta}{\overline{t}_0\rho_0^{-1}H(\overline{t}_0)}.
\end{equation}
Now by a priori information we have $d_m\leq C\rho_0$ where $C$ depends on $E$ and $M$ only. Therefore by \eqref{trivially} we have trivially

\begin{equation}
\label{1-36r}
d_m\leq C\rho_0\leq C\rho_0\left(\frac{e^{\mathcal{F}(\overline{t}_0)}\eta}{\overline{t}_0\rho_0^{-1}H(\overline{t}_0)}\right)^{1/K_0}.
\end{equation}
Therefore in both the cases we have \eqref{1-39r}.

Now we prove \eqref{2-35r}. Let us denote by $d=d_{\mathcal{H}}\left(\overline{\Omega}_1,\overline{\Omega}_2\right)$.
With no loss of generality, let $\overline{y}\in\overline{\Omega_1}
\setminus\overline{\Omega_2}$ be such that $\hbox{dist}
(\overline{y}, \overline{\Omega_2})=d$.
Since in general $\overline{y}$ needs not to belong to
$\partial\Omega_1$, \cite{A-B-R-V} it is necessary to analyze various different cases
separately. Denoting by $h=\hbox{dist}(\overline{y},\partial\Omega_1)$, let us
distinguish the following three cases:

 i) $h\leq{d \over 2}$,

ii) $h>{d \over 2}$, $h>{d_0 \over 2}$,

iii) $h>{d \over 2}$, $h\leq{d_0 \over 2}$,

\noindent where $d_0$ is the number introduced in Proposition \ref{Pr4.8}.

If case i) occurs, taking $\overline{z}\in\partial\Omega_1$ such that
$|\overline{y}-\overline{z}|=h$, we have that $\hbox{dist}(\overline{z},\overline{
\Omega}_2)\geq d-h\geq{d \over 2}$, so that
$d\leq 2d_m$ and \eqref{2-35r} follows by \eqref{1-39r}.

Let us now consider case ii). Let us denote
\begin{equation}
\label{3-40r}
d_1=\min\left\{{d \over 2},{s_0d_0 \over 4}\right\}.
\end{equation}
where $s_0$, $s_0\in(0,1)$, is defined in Theorem \ref{5-115} and depends on $\lambda$ and $\Lambda$ only. We have that

\begin{equation}
\label{3-40r}
B_{d_0/2}(\overline{y})\subset\Omega_1 \quad \mbox{ and } B_{d_1}(\overline{y})\subset\Omega_1\setminus\overline{\Omega}_2.
\end{equation}
Now by applying Theorem \ref{5-115} with $r_0=d_1$ and $\rho=\frac{s_0\rho_0}{2}$ we have

\begin{gather*}
\sup_{t\in[0,t_0-\lambda\rho_0]}\left\Vert u_1(\cdot,t) \right\Vert_{L^2\left(B_{d_1}(\overline{y})\right)} \leq \\ \nonumber
 \leq C\left(\rho_0^{-1}t_0H(t_0)\right)\left(\theta_2\log \left(\frac{\rho_0^{-1}t_0H(t_0)}{\eta}\right)\right)^{-1/6},
\end{gather*}
where
\begin{equation*}
\label{theta1}
\theta_2=\frac{1}{C\log (\rho_0/d_m)},
\end{equation*}
and $C$ depends on $\lambda$ and $\Lambda$ only.

Now proceeding exactly as in the proof of \eqref{2-38r} we have
\begin{equation}
\label{1-42r}
d_1\leq \rho_0\left(\frac{\eta}{\rho_0^{-1}t_0H(t_0)}\right)^{1/K_0},
\end{equation}
where $K_0$ is defined by \eqref{1-38r} (perhaps with a different value of constant C).

Now, if

$$\rho_0\left(\frac{\eta}{\rho_0^{-1}t_0H(t_0)}\right)^{1/K_0}<{s_0d_0 \over 4}$$
then by \eqref{1-42r} we have $d_1<{s_0d_0 \over 4}$, hence $d_1={d \over 2}$. Therefore we get
\begin{equation}
\label{1-43r}
d=2d_1\leq 2\rho_0\left(\frac{\eta}{\rho_0^{-1}t_0H(t_0)}\right)^{1/K_0}.
\end{equation}
if, instead, we have
$$\rho_0\left(\frac{\eta}{\rho_0^{-1}t_0H(t_0)}\right)^{1/K_0}\geq {s_0d_0 \over 4},$$
we have trivially

\begin{equation}
\label{1-44r}
d\leq C\rho_0\leq \frac{4C\rho_0^2}{s_0d_0}\left(\frac{\eta}{\rho_0^{-1}t_0H(t_0)}\right)^{1/K_0},
\end{equation}
where $C$ depends on $E$ and $M$ only.

If case iii) occurs we have in particular that $d<d_0$, hence by Proposition \ref{Pr4.8} we have $d\leq cd_m$ and by \eqref{1-39r} the thesis follows again.
\end{proof}

\begin{cor}\label{CorStimaforte}
Let $t_{\ast}$ be defined in Proposition \ref{stimaforte} and let $t_0\geq t_{\ast}$ be fixed. We have for every $\varepsilon\in(0,\overline{\varepsilon}]$, $\overline{\varepsilon}$ is defined in Proposition \ref{prop-20r},
\begin{equation}
\label{1-3f}
d_{\mathcal{H}}\left(\overline{\Omega}_1,\overline{\Omega}_2\right)\leq \rho_0\omega_1(\varepsilon,t_0),
\end{equation}
where
\begin{equation}
\label{omega-1}
\omega_1(\varepsilon,t_0):= C\left(\frac{\omega(\varepsilon,t_0)}{\overline{t}_0\rho_0^{-1}H(\overline{t}_0)}\right)^{1/K_0},
\end{equation}
$\omega(\varepsilon,t_0)$ is defined by \eqref{omega-enunciato-20r} and $C$ on $E, M, \lambda$, $\Lambda$ and $F$ only and $K_0$ is defined in \eqref{1-38r}.
\end{cor}
\begin{proof}
 Inequality \eqref{1-3f} is an immediate consequence of Proposition \ref{prop-20r} and Proposition \ref{stimaforte}
 \end{proof}

\subsection{Step 3} \label{step3}

Now we conclude the proof of the main Theorem.

Let $t_0\geq t_{\ast}$ be fixed and let $d_0$ be defined in Proposition \ref{Pr4.8} and let $s\in (0,\frac{d_0}{\rho_0}]$ be a number that we shall choose later. Denote by $$\epsilon(s)=\sup\left\{\varepsilon\in(0,\overline{\varepsilon}]:\omega_1(\varepsilon,t_0)\leq s \right\}.$$

By Proposition \ref{Pr4.8} we have that, for every $s\in (0,\frac{d_0}{\rho_0}]$ and every $\varepsilon\in(0,\epsilon(s)]$, $\partial\Omega_1$ and $\partial\Omega_2$ are relative graphs, moreover $G$ is equal to $\Omega_1\cap\Omega_2$ and is a domain of Lipschitz class with constants $CE$ and $\rho_0/C$ where $C\geq 1$ depends on $E$ only.

We have $$\partial\left(\Omega_1\setminus G\right)\subset\Gamma_1^{(i)}\cup\left(\Gamma_2^{(i)}\cap\partial G\right).$$
Denote by $u=u_1-u_2$. By Schwarz inequality, energy inequality, \eqref{1-s15a}, \eqref{1-s15b} and recalling that $u_2=0$ on $\Gamma_2^{(i)}$ we have, for any $t\in (0,t_0]$,

\begin{gather}
\label{OmG-15-1}
\rho_0^{-n}\int_{\Omega_1\setminus G}u_1^2(x,t)dx\leq t_0\rho_0^{-n}\int_0^{t_0}\int_{\Omega_1\setminus G}\left\vert \partial_\xi u_1(x,\xi)\right\vert^2 dxd\xi \leq \\ \nonumber
\leq C(t_0\rho_0^{-1})^{5/2}\left(H(t_0)\right)^{3/2}\left\Vert u\right\Vert_{L^{\infty}((\Gamma_2^{(i)}\cap\partial G)\times[0,t_0])}^{1/2},
\end{gather}
where $C$ depends on $\alpha, E, M, \lambda$ and $\Lambda$ only.

Let $P\in\partial G$, without restriction we may assume that $P\equiv 0$. By \eqref{1-3f} and Proposition \ref{Pr4.8} we have that if $s\in (0,\frac{d_0}{\rho_0}]$ and $\varepsilon\in(0,\epsilon(s)]$ then there exist $\varphi_{1},\varphi_{2}\in C^{1,1}\left(B_{r_0}^{\prime }\left(0\right)\right)$, where $\frac{r_0}{\rho_0}\leq 1$ depends on $E$  only, satisfying the following conditions

\begin{subequations}
\label{relgraph14-12}
\begin{equation}
\label{relgraph14-12-b}
\left \Vert\varphi_{i}\right \Vert_{C^{1,1}\left(
 B_{r_{0}}^{\prime }\left( 0\right) \right) }\leq E\rho_0,
\end{equation}
\begin{equation}
\label{relgraph-c}
\Omega_i\cap B_{r_{0}}\left(0\right) =\left\{x\in B_{r_{0}}\left(
0\right) :x_n>\varphi_{i}\left( x^{\prime }\right)\right\}\mbox{, } i=1,2.
\end{equation}
\end{subequations}
It is not restrictive to assume that
\begin{equation}
\label{relgraph-a}
\varphi _{1}\left( 0\right) =\left \vert \nabla_{x'}\varphi _{1}\left(
0\right) \right \vert=0 \quad\mbox{, }  \varphi _{2}(0)\leq 0.
\end{equation}

Now, Let us denote by $\varphi=\max\{\varphi _{1}, \varphi _{2}\}$ and by $d_1=\min\{d_0,r_0\}$ By \eqref{1-3f} and \eqref{4.361-b} we have (we fix $\alpha=1/2$), for every $s\in (0,\frac{d_1}{\rho_0}]$ and every $\varepsilon\in(0,\epsilon(s)]$,

\begin{equation}
\label{3-3f}
\left \Vert \nabla_{x'}\varphi\right\Vert_{L^{\infty}\left(B'_{s\rho_0}\right)}\leq L_{s}:=C_{\ast} s^{1/4}
\end{equation}
where $C_{\ast}$ depends on $E$ only.

For any $s\in (0,\frac{d_1}{\rho_0}]$ let us introduce the following notation
\begin{equation}
\label{4f-14-12-1}
T_s:=\max\left\{T(\epsilon(s)),2t_0\right\},
\end{equation}
where $T(\varepsilon)$ is defined in \eqref{3-epsilon20r},

\begin{equation}
\label{1-4f}
\gamma=\arctan \frac{1}{L_s}.
\end{equation}

\noindent Moreover let $\gamma_1,\gamma_2$ two numbers such that $0<\gamma_1<\gamma_2<\gamma<\frac{\pi}{2}$ that we shall choose later and let
\begin{subequations}
\label{5f}
\begin{equation}
\label{4-5f}
\chi=\frac{1-\sin\gamma_2}{1-\sin\gamma_1},
\end{equation}
\begin{equation}
\label{2-5f}
l_1=\frac{sL_s\rho_0/2}{1+\sin\gamma},
\end{equation}
\begin{equation}
\label{5-5f}
l_k=\chi^{k-1}l_1 \quad\mbox{,  } k\in \mathbb{N},
\end{equation}
\begin{equation}
\label{6f-1}
w_k=l_ke_n \quad\mbox{, } k\in \mathbb{N},
\end{equation}
\begin{equation}
\label{6f-1}
R_k=l_k\sin\gamma \mbox{, }\quad  \rho_k=l_k\sin\gamma_2 \mbox{, }\quad r_k=l_k\sin\gamma_1 \mbox{,  }\quad  k\in \mathbb{N}.
\end{equation}
\end{subequations}
It is easy to check that denoting by $\mathcal{C}$ the cone $$\mathcal{C}=\left\{x\in\mathbb{R}^n:L_s\left\vert x'\right\vert\leq x_n\leq \frac{sL_s\rho_0}{2}\right\}$$
we have

\begin{equation}
\label{6f-2}
B_{r_{k+1}}(w_{k+1})\subset B_{\rho_{k}}(w_{k})\subset B_{R_{k}}(w_{k})\subset \mathcal{C}\subset G \quad\mbox{, for every }  k\in \mathbb{N}
\end{equation}
and

\begin{equation}
\label{12f}
\mbox{dist}\left(B_{r_{1}}(w_{1}), \partial G\right)\geq \frac{1}{2}\rho_0s h,
\end{equation}
where

\begin{equation}
\label{1-13f}
h=\frac{\sin\gamma-\sin\gamma_1}{1+\sin\gamma}.
\end{equation}

Let $T\geq T_s$ be a number that we will choose. For any positive number $\mu$ such that $\mu T^2\geq 1$ and $\tau\in(0,T/2]$ denote by $U^{(\tau)}_{\mu}$ the FBI transform of $u$ defined by

\begin{equation}
\label{defFBI-u}
U^{(\tau)}_{\mu}(x,y)=\sqrt{\frac{\mu}{2\pi}}\int^{T}_0e^{-\frac{\mu}{2}(iy+\tau-t)^2}u(x,t)dt, \quad \hbox{for } (x,y)\in G\times\mathbb{R}.
\end{equation}

Let $\kappa_0\leq 1$ such that $G_r$ is connected for every $r\in(0,\kappa_0\rho_0]$ \cite{A-R-R-V}. Let $\kappa_1=\min\{\frac{d_1}{\rho_0},\kappa_0\}$.
Arguing as in Proposition \ref{prop-20r} we have by \eqref{3-141}, for every $s\in (0,\kappa_1]$ and every $\varepsilon\in(0,\epsilon(s)]$,

\begin{gather}
\label{2-2f}
\left\Vert U^{(\tau)}_{\mu}\right\Vert_{L^{2}\left(\widetilde{B}_{r_1}\left(\widetilde{w}_1\right)\right)}\leq \\ \nonumber \leq CT\rho_0^{-1}\left(H(T)+1\right)e^{\mu (s\rho_0)^2/2}\left(e^{2\mu\rho_0^2}\left(e^{-\mu T^2/10}+\varepsilon_1\right)^{\vartheta_2}\right)^{\vartheta_2^{(hs/2)^{-n}}},
\end{gather}
where $\vartheta_2\in(0,1)$ is the same exponent of inequality \eqref{1-14r}, $\vartheta_2, C$ depend on $E, M, \lambda$ and $\Lambda$ only and
\begin{equation}
\label{eta-11r-15}
\varepsilon_1=\frac{(\mu T^2)^{1/4}\varepsilon}{(H(T)+1)T\rho_0^{-1}}.
\end{equation}

\bigskip

Now we apply inequality \eqref{1-58r} when $\widetilde{r}_1=r_k, \widetilde{r}_2=\rho_k, \widetilde{r}_3=R_k $ and $x_0=w_k$, $k\in\mathbb{N}$.

Let us denote by

\begin{gather}
\label{10f-alfagammak}
\alpha_k=e^{-\mu T^2/10}+\frac{e^{-\mu R_k^2/2}
\left\Vert U^{(\tau)}_{\mu}\right\Vert_{L^{2}\left(\widetilde{B}_{r_k}(\widetilde{w}_k)\right)}}
{(H(T)+1)T\rho_0^{-1}}.
\end{gather}

Taking into account \eqref{6f-2} we have

\begin{equation}
\label{10f-alfagammak}
\alpha_{k+1}\leq \widetilde{C}_0e^{\frac{\mu}{2}\left(R_k^2-R_{k+1}^2\right)}\alpha_k^{\widetilde{\vartheta}_0} \mbox{, for every } k\in\mathbb{N},
\end{equation}
where

\begin{equation}
\label{varthetatilde}
\widetilde{\vartheta}_0=\frac{\rho_1^{-\beta_1}-\left[(1-\delta)R_1\right]^{-\beta_1}}
{\left[(1-2\delta)r_1\right]^{-\beta_1}-
\left[(1-\delta)R_1\right]^{-\beta_1}},
\end{equation}

\begin{equation}
\label{deltatilde}
0<\delta\leq\frac{R_1-\rho_1}{2R_1},
\end{equation}

\begin{equation}
\label{citilde}
\widetilde{C}_0=
C\frac{e^{C\left[(\rho_1R_1^{-1})^{-\beta_1}-(1-\delta)^{-\beta_1}\right]}}{\delta^{4}},
\end{equation}
$\beta_1$ has been introduced in Theorem \ref{new three sphere} and $C$ depends on $E, M, \lambda$ and $\Lambda$ only.

Notice that $$R_k^2-R_{k+1}^2=\chi^{2k}R_1^2(\chi^{-2}-1).$$
By iterating \eqref{10f-alfagammak} we get
\begin{equation}
\label{1-12f}
\alpha_{k+1}\leq (C\widetilde{C}_0)^{1/(1-\widetilde{\vartheta}_0)}\left(e^{\mu R_1^2 A_k/2}\alpha_1\right)^{\widetilde{\vartheta}_0^k} \mbox{, }\quad k\in\mathbb{N},
\end{equation}
where
\begin{equation}
\label{12f-1new}
A_k=(\chi^{-2}-1)(\chi^2\widetilde{\vartheta}_0^{-1})\frac{1-(\chi^2\widetilde{\vartheta}_0^{-1})^k}{1-(\chi^2\widetilde{\vartheta}_0^{-1})}
\mbox{, }\quad k\in\mathbb{N}.
\end{equation}

Let $\kappa_2=\min\{\kappa_1, 2\left(\left\vert\log_4\vartheta_2\right\vert\right)^{1/n}\}$ and taking into account that, by \eqref{1-13f}, $h\leq 1$, from \eqref{2-2f} and \eqref{1-12f} we get that for every $s\leq \kappa_2$
the following inequality holds true

\begin{gather}
\label{14f-1new}
\frac{\left\Vert U^{(\tau)}_{\mu}\right\Vert_{L^{2}\left(\widetilde{B}_{r_{k+1}}(\widetilde{w}_{k+1})\right)}}
{(H(T)+1)T\rho_0^{-1}}\leq \\ \nonumber
\leq (C\widetilde{C}_0)^{1/(1-\widetilde{\vartheta}_0)}\left(e^{\mu A^{(1)}_{s,k}}\varepsilon_1^{\vartheta_1\vartheta_2^{(sh/2)^{-n}}}+e^{\mu A^{(2)}_{s,k}}\right)^{\widetilde{\vartheta}_0^{k}}\mbox{, } k\in\mathbb{N},
\end{gather}
where
\begin{equation}
\label{B-1-sk}
A^{(1)}_{s,k}=\frac{1}{2}\left(A_k+(\chi^2\widetilde{\vartheta}_0^{-1})\right)R_1^2+\rho_0^2
\quad\mbox{, } k\in\mathbb{N},
\end{equation}
\begin{equation}
\label{B-2-sk}
A^{(2)}_{s,k}=A^{(1)}_{s,k}-\frac{1}{10}T^2 \vartheta_2^{1+(sh/2)^{-n}}
\mbox{, } k\in\mathbb{N},
\end{equation}
and $C$ depends on $E, M, \lambda$ and $\Lambda$ only.

Since we need that $A^{(1)}_{s,k}$ is bounded for $k\in \mathbb{N}$ we search for which $s\in(0,\kappa_2]$ we have
\begin{equation}
\label{3-14f}
\chi^2\widetilde{\vartheta}_0^{-1}<1,
\end{equation}

Let $\varsigma,a,b,q\in(0,1)$ three numbers that we will fix later on and let
\begin{equation}
\label{seni-16f}
\sin\gamma_1=1-\varsigma,\quad\sin\gamma_2=1-a\varsigma,\quad \sin\gamma=1-ab\varsigma
\end{equation}
and

\begin{equation}
\label{delta-16f}
\delta=q\left(\frac{R_1-\rho_1}{2R_1}\right)=\frac{q}{2}\frac{a(1-b)\varsigma}{1-ab\varsigma},
\end{equation}
by \eqref{5f} and \eqref{varthetatilde} we have respectively

\begin{equation}
\label{chi}
\chi=a,
\end{equation}
and

\begin{equation}
\label{1-16f}
\widetilde{\vartheta}_0=\frac{a(1-b)(1-q/2)}{1-ab+qa(1-b)/2}+o(1)\quad\mbox{, as } \varsigma\rightarrow 0.
\end{equation}

\bigskip

In order that \eqref{3-14f} is satisfied it is enough that

\begin{equation}
\label{19f}
a^2<\frac{a(1-b)(1-q/2)}{1-ab+qa(1-b)/2},
\end{equation}
for instance if we choose
\begin{equation}
\label{q-a-b}
q=\frac{1}{2}\mbox{ , }\quad a=\frac{1}{4} \mbox{, }\quad b=\frac{1}{3}
\end{equation}
then \eqref{19f} is satisfied and we have
\begin{equation}
\label{21f}
\chi^2\widetilde{\vartheta}_0^{-1}=\frac{23}{48}+o(1)\quad\mbox{, as } \varsigma\rightarrow 0.
\end{equation}
By \eqref{21f} we have that there exists $\varsigma_0>0$ such that if $0<\varsigma\leq\varsigma_0$ then
\begin{equation}
\label{21f-1}
\chi^2\widetilde{\vartheta}_0^{-1}\leq\frac{1}{2}.
\end{equation}
Let
\[\varsigma_1=\left(1-\left(1+C_{\ast}\kappa_2^{1/2}\right)^{-1/2}\right)^{1/2},\]
where $C_{\ast}$ is defined in \eqref{3-3f} and depends on $E$ only.

Now, let us fix $\varsigma=\overline{\varsigma}:=\min\left\{\varsigma_0,\varsigma_1,\frac{1}{4}\right\}$ and denote by $\overline{\gamma}_1,\overline{\gamma}_2,\overline{\gamma}$ the numbers belonging to $(0,\frac{\pi}{2})$ such that
\begin{equation}
\label{seni-16f-1}
\sin\overline{\gamma}_1=1-\overline{\varsigma},\quad\sin\overline{\gamma}_2=1-\frac{1}{4}\overline{\varsigma},\quad \sin\overline{\gamma}=1-\frac{1}{12}\overline{\varsigma}
\end{equation}
and denote by
\begin{equation}
\label{esse-varsigma}
\overline{s}=\frac{1}{C_{\ast}^4}\left(\left(1-\overline{\varsigma}/12\right)^{-4}-1\right).
\end{equation}
Notice that \eqref{3-3f}, \eqref{esse-varsigma} and the third equality of \eqref{seni-16f-1} imply that equality \eqref{1-4f} is satisfied. Namely we have
$$\overline{\gamma}=\arctan\frac{1}{L_{\overline{s}}}.$$

Now for any quantity $g$ introduced in \eqref{5f}, \eqref{1-13f}, \eqref{varthetatilde} and \eqref{delta-16f} we denote by $\overline{g}$ the value of such a quantity
when $s=\overline{s}$ or, equivalently, when $\varsigma=\overline{\varsigma}$. In particular we have

\begin{equation}
\label{delta-21f}
\overline{\delta}=\frac{\overline{\varsigma}}{24-2\overline{\varsigma}},
\end{equation}
\begin{equation}
\label{h-varsigma}
\overline{h}=\frac{\sin\overline{\gamma}-\sin\overline{\gamma}_1}{1+\sin\overline{\gamma}}=\frac{2\overline{s}}{24-3\overline{s}},
\end{equation}

\begin{equation}
\label{thetabar}
\overline{\widetilde{\vartheta}}_0=\frac{\left(\frac{\sin\overline{\gamma}_2}{\sin\overline{\gamma}}\right)^{-\beta_1}-
(1-\overline{\delta})^{-\beta_1}}
{\left((1-2\overline{\delta})\frac{\sin\overline{\gamma}_1}{\sin\overline{\gamma}})\right)^{-\beta_1}-(1-\overline{\delta})^{-\beta_1}}.
\end{equation}
By \eqref{12f-1new}, \eqref{21f-1}, \eqref{B-1-sk} and \eqref{esse-varsigma} we have
\begin{equation}
\label{B-1-sk-1new}
A^{(1)}_{\overline{s},k}\leq2\rho_0^2
\quad\mbox{, } k\in\mathbb{N}.
\end{equation}

Let
\begin{equation}
\label{T-tilde}
\widetilde{T}= \max\left\{\left(40\vartheta_2^{-1-(\overline{s}\overline{h}/2)^{-n}}\right)^{1/2}\rho_0,T_{\overline{s}}\right\}.
\end{equation}
By \eqref{B-2-sk}, \eqref{B-1-sk-1new} and \eqref{T-tilde}
we have, for every $T\geq\widetilde{T}$,
\begin{equation}
\label{B-2-sk-1new}
A^{(2)}_{\overline{s},k}\leq-\frac{1}{20}T^2 \vartheta_2^{1+(\overline{s}\overline{h}/2)^{-n}}
\mbox{, } k\in\mathbb{N},
\end{equation}
and $C$ depends on $E, M, \lambda$ and $\Lambda$ only.

By \eqref{B-1-sk-1new}, \eqref{B-2-sk-1new} we have, for every $T\geq\widetilde{T}$,

\begin{gather}
\label{3-29f}
\left\Vert U^{(\tau)}_{\mu}\right\Vert_{L^{2}\left(\widetilde{B}_{\overline{r}_{k+1}}(\widetilde{\overline{w}}_{k+1})\right)}
\leq \\ \nonumber
\leq \overline{C}_0(H(T)+1)T\rho_0^{-1}\left(e^{2\mu\rho_0^2 }\widetilde{\varepsilon}_1^{\delta_3}+e^{-\frac{1}{20}\mu T^2 \delta_3}\right)^{\overline{\widetilde{\vartheta}}_0^{k}}\mbox{, } k\in\mathbb{N},
\end{gather}
where
\[\overline{C}_0=(C\widetilde{C}_0)^{1/(1-\overline{\widetilde{\vartheta}}_0)},\]
\[\widetilde{\varepsilon}_1=\frac{(\mu T^2)^{1/4}\varepsilon}{(H(T)+1)(T\rho_0^{-1}+1)}\]
\[\delta_3=\vartheta_2^{1+(\overline{s}\overline{h}/2)^{-n}}\]
and $C$ depends on $E, M, \lambda$ and $\Lambda$ only.

Denote by
\[d_1=\overline{l}_1\left(1-\sin\overline{\gamma}_1\right)\mbox{ , } d_k=\chi^{k-1}d_1 \mbox{, for every } k\in\mathbb{N}\]
here, we recall that by \eqref{chi} and \eqref{q-a-b} we have $\chi=\frac{1}{4}$ and, by \eqref{2-5f} $\overline{l}_1=\frac{C_{}\ast \overline{s}^{3/2}\rho_0/2}{1+\sin\overline{\gamma}}$.

Let $r\in(0,d_1]$ be a number that we will choose later on. Let us denote by $\varrho=rd_1^{-1}$,
$$k_0=\min\left\{k\in\mathbb{N}: d_k\leq r\right\}$$
and
$$\delta_4=\left\vert\log_4\overline{\widetilde{\vartheta}}_0\right\vert.$$
We have
\begin{equation}
\label{2-25f}
\left\vert\log_4(\varrho/4)\right\vert\leq k_0 < \left\vert\log_4(\varrho/16)\right\vert
\end{equation}
and
\begin{equation}
\label{1-26f}
\overline{\widetilde{\vartheta}}_0^2\varrho^{\delta_4}\leq \overline{\widetilde{\vartheta}}_0^{k_0}\leq\overline{\widetilde{\vartheta}}_0\varrho^{\delta_4}.
\end{equation}

Now by applying \cite[Theorem 8.17]{GT} \eqref{1-s15} \eqref{2-59}, \eqref{1-63} we have, for every $\tau\in(0,t_0]$ and every $T\geq\widetilde{T}$,
\begin{gather}
\label{4-26f}
\left\vert u\left(0,\tau\right)\right\vert\leq \left\vert u\left(0,\tau\right)-u\left(\widetilde{\overline{w}}_{k_0+1},\tau\right)\right\vert+\left\vert u\left(\widetilde{\overline{w}}_{k_0+1},\tau\right)-U^{(\tau)}_{\mu}\left(\widetilde{\overline{w}}_{k_0+1}\right)\right\vert+\\ \nonumber
+\left\vert U^{(\tau)}_{\mu}\left(\widetilde{\overline{w}}_{k_0+1}\right)\right\vert\leq Ct_0\rho_0^{-1}H(t_0)\varrho+C\left(T\rho_0^{-1}\right)^2H(T)\left(T^2\mu\right)^{-1/2}+ \\ \nonumber
+C\varrho^{-\left(\frac{n+1}{2}\right)}(H(T)+1)T\rho_0^{-1}\left(e^{2\mu\rho_0^2 }\widetilde{\varepsilon}_1^{\delta_3}+e^{-\frac{1}{20}\mu T^2 \delta_3}\right)^{\overline{\widetilde{\vartheta}}_0^2\varrho^{\delta_4}}
\end{gather}
where $C$ depends on $E, M, \lambda$ and $\Lambda$ only.
Now we have trivially
\begin{equation}
\label{1-final}
e^{2\mu\rho_0^2 }\widetilde{\varepsilon}_1^{\delta_3}+e^{-\frac{1}{20}\mu T^2 \delta_3}\leq e^{2\mu T^2 }\varepsilon^{\delta_3}+e^{-\frac{1}{20}\mu T^2 \delta_3}.
\end{equation}
Hence, if

\begin{equation*}
\varepsilon\leq e^{-\left(2/\delta_3+1/20\right)}
\end{equation*}
then we choose
\begin{equation*}
\label{1-26f}
\mu=\frac{1}{T^2}\frac{\delta_3|\log \varepsilon|}{2+\delta_3/20}
\end{equation*}
and by \eqref{4-26f} and \eqref{1-final} we have

\begin{gather}
\label{2-final}
\left\vert u\left(0,\tau\right)\right\vert\leq Ct_0\rho_0^{-1}H(t_0)\varrho+C\left(T\rho_0^{-1}\right)^2H(T)|\log \varepsilon|^{-1/2}+ \\ \nonumber
+C\varrho^{-\left(\frac{n+1}{2}\right)}(H(T)+1)T\rho_0^{-1}\varepsilon^{\delta_5\varrho^{\delta_4}}
\end{gather}
where $C$ depends on $E, M, \lambda$ and $\Lambda$ only and
$$\delta_5=\frac{\delta_3\overline{\widetilde{\vartheta}}_0^2}{40+\delta_3}.$$
Now let us choose
$$\varrho=|\log\varepsilon|^{-1/(2\delta_4)}$$
and by \eqref{2-final} we have
\begin{gather}
\label{3-final}
\left\vert u\left(0,\tau\right)\right\vert\leq C\left(T\rho_0^{-1}\right)^2(H(T)+1)|\log \varepsilon|^{-1/2}
\end{gather}
where $C$ depends on $E, M, \lambda$ and $\Lambda$ only.

\noindent Otherwise, if
\begin{equation*}
\varepsilon\geq e^{-\left(2/\delta_3+1/20\right)}
\end{equation*}
then by \eqref{1-s15b} we have trivially

\begin{gather}
\label{4-final}
\left\vert u\left(0,\tau\right)\right\vert\leq Ct_0\rho_0^{-1}H(t_0)\leq Ce^{\left(2/\delta_3+1/20\right)}t_0\rho_0^{-1}H(t_0)\varepsilon.
\end{gather}
where $C$ depends on $E, M, \lambda$ and $\Lambda$ only.
By \eqref{3-final} and  \eqref{4-final} we have, for $0<\varepsilon<e^{-1}$ and every $T\geq\widetilde{T}$

\begin{gather}
\label{5-final}
\left\Vert u\right\Vert_{L^{\infty}((\Gamma_2^{(i)}\cap\partial G)\times[0,t_0])}\leq C\left(T\rho_0^{-1}\right)^2(H(T)+1)|\log \varepsilon|^{-1/2}
\end{gather}
where $C$ depends on $E, M, \lambda$ and $\Lambda$ only.
By \eqref{5-final} and \eqref{OmG-15-1} we have

\begin{gather}
\label{OmG-15-1-final}
\sup_{t\in[0,t_0]}\left(\rho_0^{-n}\int_{\Omega_j\setminus G}u^2_j(x,t)dx\right)\leq
\\ \nonumber
\leq C(t_0\rho_0^{-1}+1)^{5/2}\left(H(t_0)\right)^{3/2}\left(T\rho_0^{-1}\right)(H(T)+1)^{1/2}|\log \varepsilon|^{-1/4}
\end{gather}
where $C$ depends on $E, M, \lambda$ and $\Lambda$ only.

Now we fix $t_0=t_{\ast}+\lambda\rho_0$ and $T=\widetilde{T}$ and by \eqref{OmG-15-1-final} and Proposition \ref{stimaforte} we have

\begin{equation}
\label{2-35r-final}
d_{\mathcal{H}}\left(\overline{\Omega}_1,\overline{\Omega}_2\right)\leq K_1\rho_0|\log \varepsilon|^{-1/(8K_0)},
\end{equation}
where
\begin{equation*}
\label{1-38r-final}
K_0=e^{\mathcal{F}(t_{\ast})},
\end{equation*}

$$\overline{t}_0=t_0-\lambda\rho_0,$$
$$K_1=C\left(\frac{H(\widetilde{T})}{
H(t_{\ast})}\right)^{1/(8K_0)},$$
where $\mathcal{F}(t_{\ast})$ is defined by \eqref{F-CALL} and $C$ depends on $E, M, \lambda$ and $\Lambda$ only.$\square$

\section{Appendix} \label{appendix}

\subsection{Proof of Theorem \ref{boundary-Colombini}}\label{Regularity-proof}

 Theorem \ref{boundary-Colombini} is a straightforward consequence of Theorem \ref{Colombini} below and of standard results concerning the extension of function

\begin{theo}\label{Colombini}
Let $\Omega$ be a bounded domain of $\mathbb{R}^n$ that satisfies \eqref{1-138}. Let $A(x)$ be a real-valued symmetric $n\times n$ matrix satisfying \eqref{1-65}. Let $m:=\left[\frac{n+2}{4}\right]$. Assume that $\partial^k_t F\in L^{\infty}(\Omega\times(0,T))$ for every $k\in\{0,\cdots, 2m+2\}$ and let $u\in\mathcal{W}\left([0,T];\Omega\right)$ be the solution to the problem
\begin{equation}
\label{0-s7}
\left\{\begin{array}{ll}
\partial^2_{t}u-\mbox{div}\left(A(x)\nabla_x u\right)=F(x,t), \quad \hbox{in } \Omega\times [0,T],\\[2mm]
u=0 \quad \hbox{on } \partial\Omega\times [0,T]\\[2mm]
u(\cdot,0)=\partial_tu(\cdot,0)=0, \quad \hbox{in } \Omega.
\end{array}\right.
\end{equation}
Let $\alpha\in(0,1)$. Then for every $t\in[0,T]$ we have $u(\cdot,t)\in C^{1,\alpha}(\Omega)$ and the following inequalities hold true
\begin{subequations}
\label{2-s12}
\begin{equation}
\label{2-s12a}
\left\Vert \partial^2_tu(\cdot,t)\right\Vert_{L^{\infty}(\Omega)}\leq C\rho_0^{-2}\left(\rho^{2m+3}_0TF_{2m+2}+\sum^m_{j=0}\rho^{2j+2}_0F_{2j}\right)\quad\hbox{, }
\end{equation}
\begin{equation}
\label{2-s12b}
 \left\Vert u(\cdot,t)\right\Vert_{C^{1,\alpha}(\Omega)}\leq C\left(\rho^{2m+3}_0TF_{2m+2}+\sum^m_{j=0}\rho^{2j+2}_0F_{2j}\right)\quad\hbox{, }
\end{equation}
\end{subequations}
where $F_j:=\left\Vert \partial^{j}_tF\right\Vert_{L^{\infty}(\Omega\times[0,T])}$ for every $j\in \mathbb{N}\cup\{0\}$
and $C$ depends on $\alpha, n, E, M,\lambda$ and $\Lambda$ only.
\end{theo}

In order to prove Theorem \ref{Colombini} we use  propositions \ref{prop1-s1}, \ref{prop-s6} given below.

\begin{prop}\label{prop1-s1}
Assume that $\Omega$ and $A(x)$ are as in Theorem \ref{Colombini}. Let $\Omega$ be a bounded domain of $\mathbb{R}^n$ that satisfies \eqref{1-138a} and whose boundary is of class $C^{1,1}$. Let $A(x)$ be a real-valued symmetric $n\times n$ matrix satisfying \eqref{1-65}. If $f\in L^p(\Omega)$, $p\in (1,\infty)$, then the solution $v$ to the Dirichlet problem
\begin{equation}
\label{1-s2}
\left\{\begin{array}{ll}
\mbox{div}\left(A(x)\nabla v\right)=f, \quad \hbox{in } \Omega,\\[2mm]
v\in H^1_0(\Omega),
\end{array}\right.
\end{equation}
belongs to $W^{2,p}(\Omega)$ and the following estimate holds true
\begin{equation}
\label{2-s2}
\left\Vert v\right\Vert_{W^{2,p}(\Omega)}\leq C\rho_0^2 \left\Vert f\right\Vert_{L^{p}(\Omega)},
\end{equation}
where $C$ depends on $\lambda,\Lambda,E, M$ and $p$ only.
\end{prop}
\begin{proof}
The Proposition is an immediate consequence of \cite[Theorem 9.15]{GT} and \cite[Lemma 9.17]{GT}.
\end{proof}

\bigskip

\begin{prop}\label{prop-s6}
 Assume that $\Omega$ and $A(x)$ are as in Theorem \ref{Colombini}. Let $F\in L^{2}(\Omega\times(0,T))$ and let $u\in\mathcal{W}\left([0,T];\Omega\right)$ be the solution to the problem
\begin{equation}
\label{2-s6}
\left\{\begin{array}{ll}
\partial^2_{t}u-\mbox{div}\left(A(x)\nabla_x u\right)=F, \quad \hbox{in } \Omega\times [0,T],\\[2mm]
u=0 \quad \hbox{on } \partial\Omega\times [0,T],\\[2mm]
u(\cdot,0)=\partial_tu(\cdot,0)=0, \quad \hbox{in } \Omega.
\end{array}\right.
\end{equation}
Then the following inequality holds true

\begin{equation}
\label{4-s6}
\left\Vert u(\cdot,t)\right\Vert_{L^{p_0}(\Omega)}\leq C \rho_0 T \left\Vert F\right\Vert_{L^{\infty}(\Omega\times(0,T))},\quad \hbox{for every } t\in(0,T).
\end{equation}
where $p_0$ is the Sobolev imbedding exponent, namely

\begin{subequations}
\label{p_0}
\begin{equation}
\label{p_0a}
p_0=\frac{2n}{n-2} \quad \hbox{, for } n>2,
\end{equation}
\begin{equation}
\label{p_0b}
p_0\hbox{ is an arbitrary number of}\quad [2,+\infty) \quad \hbox{, for } n=2
\end{equation}
\end{subequations}
and $C$ depends on $n, E, M$ and $\lambda$ only.
\end{prop}

\begin{proof}
Let $\tau\in(0,T]$. By multiplying both the sides of first equation in \eqref{2-s6} by $\partial_tu$ and by integrating over $\Omega\times (0,\tau)$ we get \begin{gather*}
\int^{\tau}_{0}\int_{\Omega}F\partial_tu dxdt=-\frac{1}{2}\int^{\tau}_{0}\int_{\Omega} \partial_t\left(A(x)\nabla u\cdot\nabla u+\left(\partial_tu\right)^2\right)dxdt=\\
=-\frac{1}{2}\int_{\Omega}\left(A(x)\nabla u\cdot\nabla u+\left(\partial_tu\right)^2\right)dx.
\end{gather*}
Hence, denoting by
\begin{equation*}
K(\tau)=\int_{\Omega}\left(A(x)\nabla u(x,t)\cdot\nabla u(x,t)+\left(\partial_tu(x,t)\right)^2\right)dx,
\end{equation*}
we get
\begin{gather*}
K(\tau)\leq 2\int^{\tau}_{0}\int_{\Omega}\left\vert F\right\vert \left\vert\partial_tu \right\vert dxdt \leq T\int^{\tau}_{0}\int_{\Omega}F^2dxdt+\frac{1}{T}\int^{\tau}_{0}\int_{\Omega}\left(\partial_tu\right)^2dxdt\leq\\ \leq T\int^{\tau}_{0}\int_{\Omega}F^2dxdt+\frac{1}{T}\int^{\tau}_{0}K(t)dt.
\end{gather*}
By Gronwall inequality we derive the energy inequality
\begin{equation}
\label{1-s4}
K(\tau)\leq eT\int^T_0\int_{\Omega}F^2dxdt.
\end{equation}
In particular \eqref{1-s4} gives
\begin{equation}
\label{1-s5}
\int_{\Omega}\left\vert\nabla u(x,t)\right\vert^2 dx\leq e\lambda^{-1}T\int^T_0\int_{\Omega}F^2dxdt.
\end{equation}
Since $u(\cdot,t)\in H^1_0(\Omega)$, by \eqref{1-s5} and the Poincar\'{e} inequality we have
\begin{equation}
\label{3-s5}
\left\Vert u(\cdot,t)\right\Vert_{H^1(\Omega)}\leq C\rho_0 T\left\Vert F\right\Vert_{L^{\infty}(\Omega\times(0,T))}.
\end{equation}
Finally by the imbedding Sobolev theorem the thesis follows.
\end{proof}

\bigskip

\textbf{Sketch of the proof of Theorem \ref{Colombini}.}

In this sketch of the proof we skip on the question of regularity of the solution for which we refer to \cite{Co} and we focus on the proof of inequality \eqref{2-s12}.

In order to estimate $\left\Vert \partial^2_t u(\cdot,t)\right\Vert_{L^{\infty}(\Omega)}$, for every $t\in(0,T)$ we distinguish two cases: (a) $n$ is not of the type $4h+2$, $h\in\mathbb{N}\cup\{0\}$, (b) $n$ is of the type $4h+2$, $h\in\mathbb{N}\cup\{0\}$.

\textbf{Case (a)}.
Denote by $p_k$, $k\in\mathbb{N}\cup\{0\}$, the sequence such that

\begin{equation*}
\frac{1}{p_k}=\frac{1}{p_0}-\frac{2k}{n}  \quad \hbox{, for } k\in\mathbb{N}\cup\{0\}.
\end{equation*}
Notice that
\begin{equation*}
\frac{1}{p_k}=\frac{1}{p_{k-1}}-\frac{2}{n}  \quad \hbox{, for } k\in\mathbb{N}
\end{equation*}
and that
\begin{equation*}
\frac{1}{p_{m-1}}-\frac{2}{n}>0  \quad \hbox{, and } \frac{1}{p_{m}}-\frac{2}{n}<0.
\end{equation*}

Let us denote
\begin{equation*}
u^{(j)}:=\partial_t^{j}u \quad \hbox{, for every  } j\in \{0,\cdots,2m+2\}.
\end{equation*}
By \eqref{0-s7} we have, for every $j\in \{0,\cdots,2m+2\}$,

\begin{equation}
\label{1-s7}
\left\{\begin{array}{ll}
\partial^2_{t}u^{(j)}-\mbox{div}\left(A(x)\nabla_x u^{(j)}\right)=\partial_t^{j}F, \quad \hbox{in } \Omega\times [0,T],\\[2mm]
u^{(j)}=0 \quad \hbox{, on } \partial\Omega\times [0,T],\\[2mm]
u^{(j)}(\cdot,0)=\partial_tu^{(j)}(\cdot,0)=0, \quad \hbox{in } \Omega.
\end{array}\right.
\end{equation}

Observe that since $u^{(2j+2)}=\partial^2_{t}u^{(2j)}$ by \eqref{1-s7} we have that $u^{(2j)}$ is the solution to the following Dirichlet elliptic problem
\begin{equation}
\label{1-s2}
\left\{\begin{array}{ll}
\mbox{div}\left(A(x)\nabla_x u^{(2j)}\right)=u^{(2j+2)}-\partial_t^{2j}F  \quad \hbox{, in } \Omega,\\[2mm]
u^{(2j)}\in H^1_0(\Omega),
\end{array}\right.
\end{equation}
hence by Proposition \ref{prop1-s1} we have, for every $j\in \{1,\cdots,m\}$ and $t\in(0,T)$,
\begin{equation}
\label{2-s2}
\left\Vert u^{(2j)}(\cdot,t)\right\Vert_{W^{2,p_{m-j}}(\Omega)}\leq C\rho_0^2\left(\left\Vert u^{(2j+2)}(\cdot,t)\right\Vert_{L^{p_{m-j}}(\Omega)}+F_{2j}\right),
\end{equation}
where $C$ depends on $\lambda,\Lambda,E$ and $M$ only. Hence by Sobolev imbedding theorem we get, for every $j\in \{2,\cdots,m\}$ and $t\in(0,T)$,
\begin{equation}
\label{4-s8}
\left\Vert u^{(2j)}(\cdot,t)\right\Vert_{L^{p_{m-j+1}}(\Omega)}\leq C_0\rho_0^2\left(\left\Vert u^{(2j+2)}(\cdot,t)\right\Vert_{L^{p_{m-j}}(\Omega)}+F_{2j}\right),
\end{equation}
and
\begin{equation}
\label{4-s8-New}
\left\Vert u^{(2)}(\cdot,t)\right\Vert_{L^{\infty}(\Omega)}\leq C_0\rho_0^2\left(\left\Vert u^{(4)}(\cdot,t)\right\Vert_{L^{p_{m-2}}(\Omega)}+F_{2}\right),
\end{equation}
where $C_0\geq 1$ depends on $n,\lambda,\Lambda,E$ and $M$ only.
Now by applying Proposition \ref{prop-s6} to $u^{(2m+2)}$ we have, for every $t\in(0,T)$,
\begin{equation}
\label{4-s8-m}
\left\Vert u^{(2m+2)}(\cdot,t)\right\Vert_{L^{p_{0}}(\Omega)}\leq C_1\rho_0TF_{2m+2},
\end{equation}
where $C_1\geq 1$ depends on $n,\lambda,E$ and $M$ only.
Therefore, by iterating \eqref{4-s8} and by \eqref{4-s8-New} and \eqref{4-s8-m} we get, for every $t\in(0,T)$,

\begin{equation}
\label{2-s10}
\left\Vert u^{(2)}(\cdot,t)\right\Vert_{L^{\infty}(\Omega)}\leq C_1C_0^m\left(\rho^{2m+1}_0TF_{2m+2}+\sum^m_{j=0}\rho^{2j}_0F_{2j}\right),
\end{equation}

Now since $u=u^{(0)}$, by \eqref{1-s2} and by \cite[Theorem 8.33]{GT} we get, for every $t\in(0,T)$,
\begin{equation}
\label{1-s12}
\left\Vert u(\cdot,t)\right\Vert_{C^{1,\alpha}(\Omega)}\leq C\rho_0^2\left(\left\Vert u^{(2)}(\cdot,t)\right\Vert_{L^{\infty}(\Omega)}+F_{0}\right),
\end{equation}
where $C$ depends on $\alpha, n, E, M,\lambda$ and $\Lambda$ only.
By \eqref{1-s12} and \eqref{2-s10} we obtain \eqref{2-s12} in the case a.

\textbf{Case (b)}
We consider only the case $n>2$, because if $n=2$ we can proceed similarly.

If $n$ is of the type $4h+2$, $h\in\mathbb{N}\cup\{0\}$ then inequality \eqref{2-s2} continues to hold, but by Sobolev imbedding Theorem, instead of inequality \eqref{4-s8-New} we have, for every $q\in[2,\infty)$,
\begin{equation}
\label{4-s8-New-q}
\left\Vert u^{(4)}(\cdot,t)\right\Vert_{L^{q}(\Omega)}\leq C_2\rho_0^2\left(\left\Vert u^{(6)}(\cdot,t)\right\Vert_{L^{p_{m-1}}(\Omega)}+F_{4}\right),
\end{equation}
where $C_2\geq 1$ depends on $n,\lambda,\Lambda,E, M$ and $q$ only.

Let us choose $q>\frac{n}{2}$, by applying \cite[Theorem 8.29]{GT} to $u^{(2)}(\cdot,t)$ we have
\begin{equation}
\label{2-s11}
\left\Vert u^{(2)}(\cdot,t)\right\Vert_{L^{\infty}(\Omega)}\leq C_3\rho_0^2\left(\left\Vert u^{(4)}(\cdot,t)\right\Vert_{L^{q}(\Omega)}+F_{2}\right),
\end{equation}
where $C_3\geq 1$ depends on $n,\lambda, E, M$ and $q$.

Now, by iterating \eqref{2-s2} and by using \eqref{4-s8-New-q} and \eqref{2-s11} we get
\begin{equation}
\label{4-s11}
\left\Vert u^{(2)}(\cdot,t)\right\Vert_{L^{\infty}(\Omega)}\leq C_2C_3C_0^{m-1}\left(\rho^{2m+1}_0TF_{2m+2}+\sum^m_{j=0}\rho^{2j}_0F_{2j}\right)
\end{equation}
and arguing as in the case (a) the thesis follows.$\square$

\subsection{Proof of Propositions \ref{pag56-62}, \ref{iperb-elliptic}}\label{proofFBI}

\textbf{Proof of Proposition \ref{pag56-62}}

We prove \eqref{1-60} for $j=0$, the proof for $j>0$ being the same.

By \eqref{defFBI} we have

\begin{gather*}
\sqrt{2\pi}\left\vert U^{(\tau)}_{\mu}(x,y)\right\vert=\left\vert\sqrt{\mu}\int^T_0e^{-\frac{\mu}{2}(iy+\tau-t)^2}u(x,t)dt\right\vert\leq\\
\leq \sqrt{\mu}e^{\frac{\mu}{2}y^2}\int^T_0e^{-\frac{\mu}{2}(\tau-t)^2}\left\vert u(x,t)\right\vert dt,
\end{gather*}
hence, the Schwarz inequality yields

\begin{gather*}
\sqrt{2\pi}\left\vert U^{(\tau)}_{\mu}(x,y)\right\vert
\leq \sqrt{\mu}e^{\frac{\mu}{2}y^2}\left(\int^T_0e^{-\mu(\tau-t)^2}dt\right)^{1/2}\left(\int^T_0\left\vert u(x,t)\right\vert^2 dt\right)^{1/2}\leq \\ \leq \sqrt{\mu}e^{\frac{\mu}{2}y^2}\left(\int^{+\infty}_0e^{-\mu t^2}dt\right)^{1/2}\left(\int^T_0\left\vert u(x,t)\right\vert^2 dt\right)^{1/2}\leq c\mu^{1/4} e^{\frac{\mu}{2}y^2}\left(\int^T_0\left\vert u(x,t)\right\vert^2dt\right)^{1/2}.
\end{gather*}
Now we prove \eqref{2-59}. By the change of variable $\eta=\sqrt{\mu}(t-\tau)$ we have
\begin{gather}\label{2-56}
\sqrt{2\pi}\left(U^{(\tau)}_{\mu}(x,0)-u(x,\tau)\right)=\\ \nonumber
=\sqrt{\mu}\int^T_0e^{-\frac{\mu}{2}(\tau-t)^2}u(x,t)dt-u(x,\tau)
\int^{+\infty}_{-\infty}e^{-\frac{\eta^2}{2}}d\eta=\\ \nonumber
=\int^{\sqrt{\mu}(T-\tau)}_{-\sqrt{\mu}\tau}e^{-\frac{\eta^2}{2}}u\left(x,\tau+\frac{\eta}{\sqrt{\mu}}\right)d\eta-
u(x,\tau)\int^{+\infty}_{-\infty}e^{-\frac{\eta^2}{2}}d\eta=\\ \nonumber
=\int^{\sqrt{\mu}(T-\tau)}_{-\sqrt{\mu}\tau}e^{-\frac{\eta^2}{2}}\left(u\left(x,\tau+\frac{\eta}{\sqrt{\mu}}\right)-u(x,\tau)\right)d\eta-\\ \nonumber
-u(x,\tau)\left(\int^{+\infty}_{\sqrt{\mu}(T-\tau)}e^{-\frac{\eta^2}{2}}d\eta+\int^{-\sqrt{\mu}\tau}_{-\infty}e^{-\frac{\eta^2}{2}}d\eta\right):=I_1+I_2.
\end{gather}
We begin to estimate $\left\vert I_1\right\vert$. We have

\begin{gather}
\label{1-58}
\left\vert I_1\right\vert\leq
\mu^{-1/2}\left\Vert\partial_t u(x,\cdot)\right\Vert_{L^{\infty}[0,T]}\int^{+\infty}_{-\infty} |\eta| e^{-\frac{\eta^2}{2}}d\eta \leq \\ \nonumber
\leq c\mu^{-1/2}\left\Vert\partial_t u(x,\cdot)\right\Vert_{L^{\infty}[0,T]},
\end{gather}
where $c$ is an absolute constant.

Now we estimate $\left\vert I_2\right\vert$. Taking into account that $\tau\in(0,T/2)$ we have

\begin{gather}\label{2-58}
\left\vert I_2\right\vert\leq 2\left\vert u(x,\tau)\right\vert\int^{+\infty}_{\sqrt{\mu}\tau}e^{-\frac{\eta^2}{2}}d\eta\leq \\ \nonumber
\leq 2\left\vert u(x,\tau)\right\vert e^{-\frac{\mu}{4}\tau^2}\int^{+\infty}_{\sqrt{\mu}\tau}e^{-\frac{\eta^2}{4}}d\eta\leq c e^{-\frac{\mu}{4}\tau^2}\left\vert u(x,\tau)\right\vert,
\end{gather}
where $c$ is an absolute constant.

Now, since $u(x,0)=0$ we have
\begin{gather}\label{2-58-new1}
e^{-\frac{\mu}{4}\tau^2}\left\vert u(x,\tau)\right\vert\leq \\ \nonumber \leq e^{-\frac{\mu}{4}\tau^2}\tau\left\Vert\partial_t u(x,\cdot)\right\Vert_{L^{\infty}[0,T]}\leq2e^{-1}\mu^{-1/2}\left\Vert\partial_t u(x,\cdot)\right\Vert_{L^{\infty}[0,T]}.
\end{gather}

By \eqref{2-56}, \eqref{1-58}, \eqref{2-58} and \eqref{2-58-new1} we get \eqref{2-59}.$\square$

\bigskip

\textbf{Proof of Proposition \ref{iperb-elliptic}.}
We have
\begin{gather*}
\partial_yU_{\mu}(x,y)=\sqrt{\frac{\mu}{2\pi}}\int^T_0-i\mu(iy+\tau-t)e^{-\frac{\mu}{2}(iy+\tau-t)^2}u(x,t)dt=\\
=\sqrt{\frac{\mu}{2\pi}}\int^T_0
-i\partial_t\left(e^{-\frac{\mu}{2}(iy+\tau-t)^2}\right)u(x,t)dt=\\
=-i\sqrt{\frac{\mu}{2\pi}}\left(e^{-\frac{\mu}{2}(iy+\tau-T)^2}u(x,T)-\int^T_0
e^{-\frac{\mu}{2}(iy+\tau-t)^2}\partial_t u(x,t)dt\right)
\end{gather*}
and similarly
\begin{gather}\label{pag62}
\partial^2_yU_{\mu}(x,y)=\sqrt{\frac{\mu}{2\pi}}e^{-\frac{\mu}{2}(iy+\tau-T)^2}\left(\partial_tu(x,T)-\mu(iy+\tau-T)u(x,T)\right)-\\ \nonumber
-\sqrt{\frac{\mu}{2\pi}}\int^T_0
e^{-\frac{\mu}{2}(iy+\tau-t)^2}\partial^2_t u(x,t)dt.
\end{gather}
On the other side by \eqref{1-56} we have
\begin{gather}\label{pag62-1}
\mbox{div}\left(A(x)\nabla U_{\mu}\right)=\sqrt{\frac{\mu}{2\pi}}\int^T_0
e^{-\frac{\mu}{2}(iy+\tau-t)^2}\mbox{div}\left(A(x)\nabla u\right)dt=\\ \nonumber
=\sqrt{\frac{\mu}{2\pi}}\int^T_0
e^{-\frac{\mu}{2}(iy+\tau-t)^2}\partial^2_t u(x,t)dt.
\end{gather}
By \eqref{pag62} and \eqref{pag62-1} the thesis follows.$\square$

\subsection{Proof of Theorem \ref{new three sphere}}\label{App-3sphere}

In the sequel, for seek of brevity  we omit the tilde over $r_j$, $j=1,2,3$.

First we consider the homogeneous case in which $\widetilde{f}=0$ and we assume that $r_3=1$.  In \cite[Theorem 4.5]{M-R-V1} it has been proved that
there exists a positive number $\overline{\beta}$ depending on $\lambda_0, \Lambda_0$ only such that if $\beta>\overline{\beta}$, then there exist
constants $C$, $\tau_1$ and $r_0$, ($C\geq 1$, $\tau_1\geq 1$, $0<r_0\leq 1$) depending only on $\lambda_0,\Lambda_0$
and $\beta$ such that the following estimate holds true
\begin{gather}
\label{1-49r}
\tau\int\left\vert X \right\vert^{\beta}e^{2\tau\left\vert X \right\vert^{-\beta}}|\nabla v|^2+\tau^3\int\left\vert X \right\vert^{-\beta-2}e^{2\tau\left\vert X \right\vert^{-\beta}}|v|^2\leq \\ \nonumber
\leq C\int\left\vert X \right\vert^{2\beta+2}e^{\left\vert X \right\vert^{-\beta}}
\left\vert P v \right\vert^2,
\end{gather}
for every $v\in C_0^\infty\left(
B_{r_0}\setminus\{0\}\right)$ and for every $\tau\geq\tau_1$.

On the other hand it is simple to check that there exists $\widetilde{\beta}$ depending on $\lambda_0, \Lambda_0$ only such that if $\beta\geq \widetilde{\beta}$ then $\left\vert X \right\vert^{-\beta}$ satisfies the pseudoconvexity conditions of \cite[Theorem 8.3.1]{hormanderbook1} in $B_1\setminus \overline{B}_{r_0/2}$. Therefore there exist $\tau_2\geq\tau_1$ and $C$ depending on $\lambda_0,\Lambda_0$
and $\beta$ only such that
\begin{gather}
\label{1-50r}
\tau\int e^{2\tau\left\vert X \right\vert^{-\beta}}|\nabla v|^2+\tau^3\int e^{2\tau\left\vert X \right\vert^{-\beta}}|v|^2\leq \\ \nonumber
\leq C\int e^{\left\vert X \right\vert^{-\beta}}
\left\vert Pv \right\vert^2,
\end{gather}
for every $v\in C_0^\infty\left(
B_1\setminus \overline{B}_{r_0/2}\right)$ and for every $\tau\geq\tau_2$.

Now we have trivially
\begin{subequations}
\label{50r}
\begin{equation}
\label{50ra}
\int e^{2\tau\left\vert X \right\vert^{-\beta}}|\nabla v|^2\geq \int\left\vert X \right\vert^{\beta}e^{2\tau\left\vert X \right\vert^{-\beta}}|\nabla v|^2,
\end{equation}
\begin{equation}
\label{50rb}
\int e^{2\tau\left\vert X \right\vert^{-\beta}}|v|^2\geq (r_0/2)^{\beta+2}\int\left\vert X \right\vert^{-\beta-2}e^{2\tau\left\vert X \right\vert^{-\beta}}|v|^2,
\end{equation}
\begin{equation}
\label{50rc}
\int e^{\left\vert X \right\vert^{-\beta}}
\left\vert Pv \right\vert^2 \leq \frac{1}{(r_0/2)^{2\beta+2}}\int\left\vert X \right\vert^{2\beta+2}e^{\left\vert X \right\vert^{-\beta}}
\left\vert Pv \right\vert^2,
\end{equation}
\end{subequations}
for every $v\in C_0^\infty\left(
B_1\setminus \overline{B}_{r_0/2}\right)$.
Let $\zeta\in C_0^\infty\left(B_{r_0}\right)$ such that $0\leq \zeta\leq 1$, $|\nabla \zeta|, |D^2\zeta|\leq C$ and $\zeta(X)=1$ for every $X\in B_{r_0/2}$.

Now let us denote $\beta_1:=\max\{\overline{\beta},\widetilde{\beta},1\}$ and let $v\in C_0^\infty\left(
B_{1}\setminus\{0\}\right)$. By applying  \eqref{1-49r} and \eqref{1-50r} to $\zeta v$ and $(1-\zeta)v$ respectively and taking into account \eqref{50r} we have, for $\beta\geq \beta_1$
\begin{gather}
\label{1-49r}
\tau\int\left\vert X \right\vert^{\beta}e^{2\tau\left\vert X \right\vert^{-\beta}}|\nabla v|^2+\tau^3\int\left\vert X \right\vert^{-\beta-2}e^{2\tau\left\vert X \right\vert^{-\beta}}|v|^2\leq \\ \nonumber
\leq C\int\left\vert X \right\vert^{2\beta+2}e^{\left\vert X \right\vert^{-\beta}}
\left\vert Pv \right\vert^2+C\int\left\vert X \right\vert^{2\beta+2}e^{\left\vert X \right\vert^{-\beta}}
\left(|D^2\zeta|^2v^2+|\nabla \zeta|^2|\nabla v|^2\right).
\end{gather}
Now the second term at the right hand side can be absorbed by the left hand side. Hence there exists $\tau_3\geq\tau_2$ and $C$ depending on $\lambda_0,\Lambda_0$ and $\beta$ only such that for every $v\in C_0^\infty\left(B_{1}\setminus\{0\}\right)$ and every $\tau\geq\tau_3$ the following inequality holds true
\begin{gather}
\label{1-51r}
\tau\int\left\vert X \right\vert^{\beta}e^{2\tau\left\vert X \right\vert^{-\beta}}|\nabla v|^2+\tau^3\int\left\vert X \right\vert^{-\beta-2}e^{2\tau\left\vert X \right\vert^{-\beta}}|u|^2\leq \\ \nonumber
\leq C\int\left\vert X \right\vert^{2\beta+2}e^{\left\vert X \right\vert^{-\beta}}
\left\vert Pv \right\vert^2.
\end{gather}
Now we use a standard argument to derive by \eqref{1-51r} the desired three sphere inequality.

First we observe that, by density, estimate \eqref{1-51r} holds true for every $v\in H^2_0\left(B_{1}\setminus\{0\}\right)$. Now let $u\in H^1\left(B_{1}\right)$ a solution to equation $Pu=0$. By $L^2$ regularity theorem we have that $u\in H^2_{loc}\left(B_{1}\right)$. Let $0<r_1\leq r_2<1$, $0<\delta\leq\min\left\{\frac{1-r_2}{2},\frac{1}{2}\right\}$ and let us consider a cutoff function $\eta\in C_0^2\left(
B_{1-\delta}\setminus \overline{B}_{r_1(1-2\delta)}\right)$ such that $0\leq \eta\leq 1$ and satisfying the following conditions

\begin{equation*}
\eta=1 \quad\hbox{, in } B_{1-2\delta}\setminus B_{r_1(1-\delta)},
\end{equation*}
\begin{equation*}
|\nabla\eta|\leq\frac{c}{\delta r_1}\quad\hbox{, } |D^2\eta|\leq\frac{c}{\delta^2 r_1^2} \quad\hbox{, in } B_{r_1(1-\delta)}\setminus B_{r_1(1-2\delta)}
\end{equation*}
and
\begin{equation*}
|\nabla\eta|\leq\frac{c}{\delta}\quad\hbox{, } |D^2\eta|\leq\frac{c}{\delta^2} \quad\hbox{, in } B_{1-\delta}\setminus B_{1-2\delta},
\end{equation*}
where $c$ is an absolute constant.

By \eqref{equaz-49r}, since $\widetilde{f}=0$ we have
\begin{equation*}
|P(\eta u)|\leq C\left(|\nabla\eta||\nabla u|+|P\eta||u|\right),
\end{equation*}

\begin{gather}
\label{52r-53r}
\int\left\vert X \right\vert^{2\beta+2}e^{\left\vert X \right\vert^{-\beta}}
\left\vert P(\eta u) \right\vert^2\leq Ce^{2\tau\left((1-2\delta)r_1\right)^{-\beta}}r_1^{2\beta+2} \times\\ \nonumber
\times\left[\int_{B_{r_1(1-\delta)}\setminus B_{r_1(1-2\delta)}}\left((\delta r_1)^{-2}|\nabla u|^2+(\delta r_1)^{-4}| u|^2\right)\right]+\\ \nonumber
+Ce^{2\tau\left(1-2\delta\right)^{-\beta}}\left[\int_{B_{1-\delta}\setminus B_{1-2\delta}}\left(\delta^{-2}|\nabla u|^2+\delta ^{-4}| u|^2\right)\right],
\end{gather}
where $C$ depends on $\lambda_0,\Lambda_0$ and $\beta$ only.

By applying the Caccioppoli inequality to the right hand side of  \eqref{52r-53r} and by \eqref{1-51r} we have, for every $\tau\geq\tau_3$

\begin{gather}
\label{1-54r}
\int_{B_{r_2}}\left\vert X \right\vert^{-\beta-2}e^{2\tau\left\vert X \right\vert^{-\beta}}|u\eta|^2\leq Ce^{2\tau\left((1-2\delta)r_1\right)^{-\beta}}r_1^{2\beta-2}\delta^{-4}\int_{B_{r_1}}|u|^2+\\ \nonumber
+Ce^{2\tau\left(1-2\delta\right)^{-\beta}}\delta^{-4}\int_{B_{1}}|u|^2,
\end{gather}
where $C$ depends on $\lambda_0,\Lambda_0$ and $\beta$ only.

On the other hand we have trivially
\begin{equation}
\label{54r-55r}
\int_{B_{r_2}}\left\vert X \right\vert^{-\beta-2}e^{2\tau\left\vert X \right\vert^{-\beta}}|u\eta|^2\geq r_2^{-\beta-2}e^{2\tau r_2^{-\beta}}\int_{B_{r_2}\setminus B_{r_1}}|u|^2.
\end{equation}
Now let us denote

\begin{equation}
\label{Notation55r}
\epsilon:=\left(\int_{B_{r_1}}|u|^2\right)^{1/2} \hbox{, and } K:=\left(\int_{B_{1}}|u|^2\right)^{1/2}.
\end{equation}

By \eqref{1-54r} and \eqref{54r-55r} we have for every $\tau\geq\tau_3$

\begin{gather}
\label{New1-55r}
\int_{B_{r_2}\setminus B_{r_1}}|u|^2\leq \\ \nonumber \leq C\delta^{-4}\left\{ e^{2\tau\left[\left((1-2\delta)r_1\right)^{-\beta}-r_2^{-\beta}\right]}\epsilon^2+e^{2\tau\left[\left((1-\delta)\right)
^{-\beta}-r_2^{-\beta}\right]}K^2\right\},
\end{gather}
where $C$ depends on $\lambda_0,\Lambda_0$ and $\beta$ only.

Now we add to both the side of \eqref{New1-55r} the integral $\int_{B_{r_1}}u^2dX$ and we get
\begin{equation}
\label{1-55r}
\int_{B_{r_2}}|u|^2\leq C\delta^{-4}\left\{ e^{2\tau\left[\left((1-2\delta)r_1\right)^{-\beta}-r_2^{-\beta}\right]}\epsilon^2+e^{2\tau\left[\left((1-\delta)\right)
^{-\beta}-r_2^{-\beta}\right]}K^2\right\},
\end{equation}
where $C$ depends on $\lambda_0,\Lambda_0$ and $\beta$ only.

Now denote by $\overline{\tau}$ the number
\begin{equation}
\label{taubar}
\overline{\tau}=\frac{\log (\epsilon^{-1}K)}{\left((1-2\delta)r_1\right)^{-\beta}-(1-\delta)^{-\beta}}
\end{equation}
such a number satisfies the equality
\[e^{2\overline{\tau}\left[\left((1-2\delta)r_1\right)^{-\beta}-r_2^{-\beta}\right]}\epsilon^2=e^{2\overline{\tau}\left[\left((1-\delta)\right)
^{-\beta}-r_2^{-\beta}\right]}K^2.\]

If $\overline{\tau}\geq\tau_3$ then we choose $\tau=\overline{\tau}$ in \eqref{1-55r} and we obtain

\begin{equation}
\label{1-56r}
\int_{B_{r_2}}|u|^2\leq C\delta^{-4}K^{2(1-\vartheta)}\epsilon^{\vartheta},
\end{equation}
where
\begin{equation*}
\vartheta=\frac{r_2^{-\beta}-(1-\delta)^{-\beta}}{\left[(1-2\delta)r_1\right]^{-\beta}-(1-\delta)^{-\beta}}.
\end{equation*}

On the other side if $\overline{\tau}<\tau_3$ then by \eqref{taubar} we have
\begin{equation}
(\epsilon^{-1}K)^{2\vartheta}<e^{2\tau_3[r_2^{-\beta}-(1-\delta)^{-\beta}]}
\end{equation}
hence we have trivially
\begin{equation}
\label{1-57r}
\int_{B_{r_2}}|u|^2\leq \int_{B_{1}}|u|^2=K^2=K^{2(1-\vartheta)}K^{2\vartheta}\leq e^{2\tau_3[r_2^{-\beta}-(1-\delta)^{-\beta}]}K^{2(1-\vartheta)}\epsilon^{2\vartheta}.
\end{equation}
Therefore by \eqref{1-56r} and \eqref{1-57r} we have
\begin{equation}
\label{1-56r}
\int_{B_{r_2}}|u|^2\leq C\delta^{-4}e^{2\tau_3[r_2^{-\beta}-(1-\delta)^{-\beta}]}K^{2(1-\vartheta)}\epsilon^{\vartheta}.
\end{equation}

In the nonhomogeneous case, let $u\in H^1(B_1)$ a solution to $Pu=\widetilde{f}$ and let $w$ be the solution to the Dirichlet problem

\begin{equation*}
\left\{\begin{array}{ll}
Pw=\widetilde{f}, \quad \hbox{in } B_1,\\[2mm]
w\in H^1_0(B_1),
\end{array}\right.
\end{equation*}
we have that
\begin{equation}
\label{5gennaio-1}
\int_{B_{1}}|w|^2\leq C\int_{B_{1}}|\widetilde{f}|^2,
\end{equation}
where $C$ depends on $\lambda_0$. By applying \eqref{1-56r} to the function $u-w$ and by \eqref{5gennaio-1} we obtain inequality \eqref{1-58r} when $\widetilde{r}_3=1$. Finally, by using the dilation $X\rightarrow r_3X$ the thesis follows easily.$\Box$

\bigskip

\textbf{Acknowledgment}

\bigskip

I wish to express my gratitude to Professor Victor Isakov for stimulating
discussions on the problem.

\end{document}